\documentclass[10pt]{amsart}

\usepackage{colortbl}
\usepackage{amsmath,amssymb}
\usepackage{mathrsfs}
\usepackage{graphicx}
\usepackage{amsmath}
\usepackage{dsfont}
\usepackage{multirow}
\usepackage{graphicx}
\usepackage{epstopdf}
\usepackage{color}
\usepackage{ulem}
\usepackage{hyperref}
\usepackage{enumerate}
\usepackage{caption}
\usepackage{subfig}
\usepackage{bm}

\usepackage{varioref}
\usepackage{hyperref}
\usepackage{cleveref}

\def\abs#1{\left|#1\right|}
\def\norm#1{\|#1\|}

\crefformat{figure}{Figure~#2#1#3}
\crefmultiformat{figure}{Figures~(#2#1#3)}{ and~(#2#1#3)}{, (#2#1#3)}{ and~(#2#1#3)}
\crefformat{table}{Table~#2#1#3}

\def\m{\mbox{\boldmath $m$}}

\def\e{\mbox{\boldmath $e$}}
\def\f{\mbox{\boldmath $f$}}

\def\g{\mbox{\boldmath $g$}}

\def\x{\mbox{\boldmath $x$}}

\def\vv{\mbox{\boldmath $v$}}

\def\0{\mbox{\boldmath $0$}}
\def\um{\underline{\bm m}}

\def\mmu{\mbox{\boldmath $\mu$}}

\newtheorem{thm}{Theorem}[section]
\newtheorem{lem}{Lemma}[section]
\newtheorem{algorithm}{Algorithm}[section]
\newtheorem{defn}{Definition}[section]

\theoremstyle{definition}

\theoremstyle{remark}

\numberwithin{equation}{section}

\begin{document}

\title[Convergence analysis of an implicit FDM for iLLG equation]{Convergence analysis of an implicit finite difference method for the inertial Landau-Lifshitz-Gilbert equation}

\author{Jingrun Chen}
\address{{School of Mathematical Sciences and Suzhou Institute for Advanced Research, University of Science and Technology of China, China\newline\indent
Suzhou Institute for Advanced Research, University of Science and Technology of China, Suzhou, Jiangsu 215123, China}}
\email{jingrunchen@ustc.edu.cn}
\thanks{}

\author{Panchi Li}
\address{School of Mathematical Sciences, Soochow University, Suzhou, 215006, China}
\email{lipanchi1994@163.com}
%
\author{Cheng Wang}
\address{Mathematics Department, University of Massachusetts, North Dartmouth, MA 02747, USA}
\email{cwang1@umassd.edu}
\thanks{}

\graphicspath{{implicit_figures/}}


\date{\today}

\dedicatory{}
\begin{abstract}
The Landau-Lifshitz-Gilbert (LLG) equation is a widely used model for fast magnetization dynamics in ferromagnetic materials. Recently, the inertial LLG equation, which contains an inertial term, has been proposed to capture the ultra-fast magnetization dynamics at the sub-picosecond timescale. Mathematically, this generalized model contains the first temporal derivative and a newly introduced second temporal derivative of magnetization. Consequently, it produces extra difficulties in numerical analysis due to the mixed hyperbolic-parabolic type of this equation with degeneracy. In this work, we propose an implicit finite difference scheme based on the central difference in both time and space. A fixed point iteration method is applied to solve the implicit nonlinear system. With the help of a second order accurate constructed solution, we provide a convergence analysis in $H^1$ for this numerical scheme, in the $\ell^\infty (0, T; H_h^1)$ norm. It is shown that the proposed method is second order accurate in both time and space, with unconditional stability and a natural preservation of the magnetization length. In the hyperbolic regime, significant damping wave behaviors of magnetization at a shorter timescale are observed through numerical simulations.
\end{abstract}

\maketitle


\section{Introduction}
The Landau-Lifshitz-Gilbert (LLG) equation~\cite{Gilbert1955, LandauLifshitz1935} describes the dissipative magnetization dynamics in ferromagnetic materials, which is highly nonlinear and has a non-convex constraint. Physically, it is widely used to interpret the experimental observations. However, recent experiments~\cite{PhysRevLett1996, NatComm2011_experiment, NatComm2010Ultrafastdemagnetization} confirm that its validity is limited to timescales from picosecond to larger timescales for which the angular momentum reaches equilibrium in a force field. At shorter timescales, e.g. $\sim 100\;\mathrm{fs}$, the ultra-fast magnetization dynamics has been observed~\cite{NatComm2010Ultrafastdemagnetization}. To account for this, the inertial Landau-Lifshitz-Gilbert (iLLG) equation is proposed~\cite{PhysRevLett057204, PhysRevB020410, PhysRevB172403}. As a result, the magnetization converges to its equilibrium along a locus with damping nutation simulated in~\cite{LI2021semi-implicitiLLG}, when the inertial effect is activated by a non-equilibrium initialization or an external magnetic field.

For a ferromagnet over $\Omega\in\mathds{R}^d, d = 1,2,3$, the observable states are depicted by the distribution of the magnetization in $\Omega$. The magnetization denoted by $\m(\x, t)$ is a vector field taking values in the unit sphere $\mathds{S}^2$ of $\mathds{R}^3$, which indicates that $|\m| = 1$ in a point-wise sense. In micromagnetics, the evolution of $\m$ is governed by the LLG equation. In addition to experiment and theory, micromagnetics simulations have become increasingly important over the past several decades. Therefore, numerous numerical approaches have been proposed for the LLG equation and its equivalent form, the Landau-Lifshitz (LL) equation; see~\cite{cimrak2007survey, SIAMRev2006DevelopmentLLG} for reviews and references therein. In terms of time marching, the simplest explicit methods, such as the forward Euler method and Runge-Kutta methods, were favored in the early days, while small time step size must be used due to the stability restriction~\cite{ROMEO2008464}. Of course, implicit methods avoid the stability constraint and these methods produce the approximate solution in $H^1(\Omega)$~\cite{ALOUGES2006femMMMAS, ALOUGES20121345}. However, in order to guarantee the convergence of the schemes, a step-size condition $k = O(h^2)$ must be satisfied in both the theoretical analysis and numerical simulations. To obtain the weak solution in the finite element framework, an intermediate variable $\vv$ with the definition $\vv = \partial_t\m$ representing the increment rate at current time is introduced, and to solve $\vv$ in the tangent space of $\m$ where it satisfies $\vv\cdot\m = 0$ in a point-wise sense, then the configuration at the next time step can be obtained. Directly, the strong solution can be obtained through solving the implicit mid-point scheme~\cite{bartels2006convergence} and the implicit backward Euler scheme~\cite{fuwa2012finite} using fixed-point iteration methods. By contrast, the semi-implicit methods have achieved a desired balance between stability and efficiency for the micromagnetics simulations. The Gauss-Seidel projection methods~\cite{E2000analysisProjection, LI2020109046, wang2001gauss}, the linearized backward Euler scheme~\cite{Ivan2005semiimpliictEulerAnalysis,gao2014optimal}, the Crank-Nicolson projection scheme~\cite{An2021analysisCN}, and the second order semi-implicit backward differentiation formula projection scheme~\cite{CHEN202155,Xie2020} have been developed in recent years. In practice, all these semi-implicit methods inherit the unconditional stability of implicit schemes, and achieve the considerable improvement in efficiency.

The LLG equation is a nonlinear parabolic system which consists of the gyromagnetic term and the damping term. It is a classical kinetic equation that only contains the velocity; no acceleration is included in the equation. When relaxing the system from a non-equilibrium state or applying a perturbation, it is natural that an acceleration term will be present, resulting in the inertial term in the iLLG equation. More specifically, the time evolution of $\m(\x, t)$ is described by $\partial_t\m$ and $\m\times\partial_t\m$ with the addition of an inertial term $\m\times\partial_{tt}\m$. Thus, the iLLG equation is a nonlinear system of mixed hyperbolic-parabolic type with degeneracy. To numerically study the hyperbolic behaviors of the magnetization, the first-order accuracy tangent plane scheme (TPS) and the second-order accuracy angular momentum method (AMM) are proposed in~\cite{Ruggeri2021iLLGanalysis}. The fixed-point iteration method is used for the implicit marching. These two methods aim to find the weak solution. Furthermore, a second-order accurate semi-implicit method is presented in~\cite{LI2021semi-implicitiLLG}, and $\partial_{tt}\m$ and $\partial_t\m$ are approximated by the central difference.

In this work, we provide the convergence analysis of the implicit mid-point scheme on three time layers for the iLLG equation. Subject to the condition $k \leq Ch^2$, it produces a unique second-order approximation in $H^1({\Omega_T})$. Owing to the application of the mid-point scheme, it naturally preserves the magnetization length. Moreover, we propose a fixed-point iteration method to solve the nonlinear scheme, which converges to a unique solution under the condition of $k \leq Ch^2$. Numerical simulations are reported to confirm the theoretic analysis and study the inertial dynamics at shorter timescales.

The rest of this paper is organized as follows. The iLLG equation and the numerical method are introduced in Section \ref{sec:model}. The detailed convergence analysis is provided in Section \ref{sec:analysis}. In addition, a fixed-point iteration method for solving the implicit scheme is proposed in Section \ref{sec:solving}, and the convergence is established upon the condition $k \le Ch^2$. Numerical tests, including the accuracy test and observation of the inertial effect, are presented in Section \ref{sec:experiments}. Concluding remarks are made in Section \ref{sec:conclusion}.

\section{The physical model and the numerical method}
\label{sec:model}

The intrinsic magnetization of a ferromagnetic body $\m = \m(\x, t):\Omega_T := \Omega\times(0, T)\rightarrow\mathds{S}^2$ is modeled by the conventional LLG equation:
\begin{subequations}
  \begin{align}
    &\partial_t\m = -\m\times\Delta\m + \alpha\m\times\partial_t\m, &(\x, t)\in\Omega_T,\\
    &\m(\x, 0) = \m(0), &\x\in\Omega,\\
    &\partial_{\nu}\m(\x, t) = 0, &(\x,t)\in\partial\Omega\times[0, T],
  \end{align}
\end{subequations}
where $\nu$ represents the unit outward normal vector on $\partial\Omega$, and $\alpha \ll 1$ is the damping parameter. If the relaxation starts from a non-equilibrium state or a sudden perturbation is applied, the acceleration should be considered in the kinetic equation, which is the inertial effect observed in various experiments at the sub-picosecond timescale. In turn, its dynamics is described by the iLLG equation
\begin{subequations}\label{equ:iLL-equation}
  \begin{align}
    &\partial_t\m = -\m\times(\Delta\m + \mathbf{H}_{\mathrm{e}}) + \alpha\m\times\left(\partial_t\m + \tau\partial_{tt}\m\right), &(\x, t)\in\Omega_T,\label{subequ:iLLG-equ}\\
    &\m(\x, 0) = \m(0), &\x\in\Omega,\label{subequ:initial-condition-I}\\
    &\partial_t\m(\x, 0) = 0, &\x\in\Omega,\label{subequ:initial-condition-II}\\
    &\partial_{\nu}\m(\x, t) = 0, &(\x,t)\in\partial\Omega\times[0, T],\label{subequ:boundary-condition}
  \end{align}
\end{subequations}
where $\tau$ is the phenomenological inertia parameter, and $\mathbf{H}_{\mathrm{e}}$ is a perturbation of an applied magnetic field. To ease the discussion, the external field is neglected in the subsequent analysis and is only considered in micromagnetics simulations. An additional initial condition $\partial_t\m(\x, 0) = 0$ is added, which implies that the velocity is $0$ at $t = 0$ and it is a necessary condition for the well-posedness. Then the energy is defined as
\begin{equation}
  \label{energy}
  \mathcal{E}[\m] = \frac 1{2}\int_\Omega\left( |\nabla\m|^2 - 2\m\cdot\mathbf{H}_{\mathrm{e}} + \alpha\tau\abs{\partial_t\m}^2\right)\mathrm{d}\x.
\end{equation}
For constant external magnetic fields, it satisfies the energy dissipation law
\begin{equation}
  \frac{d}{dt}\mathcal{E}[\m] = -\alpha\int_{\Omega}|\partial_t\m|^2\mathrm{d}\x \leq 0.
  \label{equ:continuous-energy-law}
\end{equation}
Therefore, under the condition of \eqref{subequ:initial-condition-II}, for almost all $T'\in(0, T)$, we have
\begin{equation}
  \frac 1{2}\int_\Omega\left( |\nabla\m(\x,T')|^2 + \alpha\tau\abs{\partial_t\m(\x,T')}^2\right)\mathrm{d}\x \leq \frac 1{2}\int_\Omega\left( |\nabla\m(\x,0)|^2\right)\mathrm{d}\x.
\end{equation}

Before the formal algorithm is presented, here the spatial difference mesh and the temporal discretization have to be stated. The uniform mesh for $\Omega$ is constructed with mesh-size $h$ and a time step-size $k > 0$ is set. Let $L$ be the set of nodes $\{\x_l = (x_i, y_j, z_k)\}$ in 3-D space with the indices $i = 0, 1, \cdots, nx, nx+1, j = 0, 1, \cdots, ny, ny+1$ and $k = 0, 1, \cdots, nz, nz+1$, and the ghost points on the boundary of $\Omega$ are denoted by $i_x = 0, nx+1$, $j_y = 0, ny+1$ and $k_z = 0, nz+1$. We use the half grid points with $\m_{i,j,k} = \m((i-\frac 1{2})h_x, (j-\frac 1{2})h_y, (k-\frac 1{2})h_z)$. Here $h_x = 1/nx$, $h_y = 1/ny$, $h_z = 1/nz$ and $h = h_x = h_y = h_z$ holds for uniform spatial meshes. Due to the homogeneous Neumann boundary condition \eqref{subequ:boundary-condition}, the following extrapolation formula is derived:
\begin{equation}
  \m_{i_x+1,j,k} = \m_{i_x,j,k}, \m_{i,j_y+1,k} = \m_{i,j_y,k}, \m_{i,j,k_z+1} = \m_{i,j,k_z},
\end{equation}
for any $1 \leq i \leq nx, 1 \leq j \leq ny, 1 \leq k \leq nz$.
Meanwhile, the temporal derivatives are discretized by the central difference, with the details stated in the following definition.
\begin{defn}
  For $\phi^{n+1} = \phi(\x, t_{n+1})$ and $\psi^{n+1} = \psi(t_{n+1})$, define
  \begin{equation*}
    d_t^+\phi^{n} = \frac{\phi^{n+1} - \phi^{n}}{k},\ \ \ d_t^-\phi^{n} = \frac{\phi^{n} - \phi^{n-1}}{k},
  \end{equation*}
  and
  \begin{equation*}
    D_t^+\psi^{n} = \frac{\psi^{n+1} - \psi^n}{k},\ \ \ D_t^-\psi^{n} = \frac{\psi^{n} - \psi^{n-1}}{k}.
  \end{equation*}
  Consequently, we denote
  \begin{equation*}
    d_t\phi^{n+1} = \frac 1{2}(d_t^+\phi^{n} + d_t^-\phi^{n}), \ \ \ D_t\psi^{n+1} = \frac 1{2}(D_t^+\psi^{n} + D_t^-\psi^{n}).
  \end{equation*}
  In particular, the second time derivative is approximated by the central difference form
  \begin{equation}
    d_{tt}\phi = \frac{\phi^{n+1} - 2\phi^{n} + \phi^{n-1}}{k^2}.
  \end{equation}
\end{defn}
Then for the initial condition \eqref{subequ:initial-condition-II}, there holds
\begin{equation}
  \m(\x_l,0) = \m(\x_l, k),\;\forall l\in L,
\end{equation}
where $L =\{(i,j,k)| i=1,\cdots,nx; j=1,\cdots,ny; k=1,\cdots,nz.\}$.
Denote $\m_h^n (n\geq 0)$ as the numerical solution. Given grid functions $\f_h,\g_h\in\ell^2(\Omega_h,\mathds{R}^3)$, we list definitions of the discrete inner product and norms used in this paper.
\begin{defn}{The discrete inner product $\langle\cdot,\cdot\rangle$ in $\ell^2(\Omega_h,\mathds{R}^3)$ is defined by}
\begin{equation}
  \langle\f_h, \g_h\rangle = h^d\sum_{l\in L}\f_h(\x_l)\cdot\g_h(\x_{l}).
\end{equation}
The discrete $\ell^2$ norm and $H^1_h$ norm of $\m_h$ are
\begin{equation}
  \|\f_h\|_2^2 =  h^d\sum_{l\in L}\f_h(\x_l)\cdot\f_h(\x_{l}),
\end{equation}
and
\begin{equation}
  \|\f_h\|_{H^1_h}^2 = \|\f_h\|_2^2 + \|\nabla_h\f_h\|_2^2
\end{equation}
with $\nabla_h$ representing the central difference stencil of the gradient operator.
\end{defn}
Besides, the norm $\|\cdot\|_{\infty}$ in $\ell^{\infty}(\Omega_h,\mathds{R}^3)$ is defined by
\begin{equation}
  \|\f_h\|_{\infty} = \max_{\l\in L}\|\f_h(\x_l)\|_{\infty}.
\end{equation}

Therefore, the approximation scheme of the iLLG equation is presented below.
\begin{algorithm}
  \label{alg:implicit-midpoint-algorithm}
  Given $\m_h^0, \m_h^1\in W^{1,2}(\Omega_h, \mathds{S}^2)$. Let $\m_h^{n-1},\m_h^{n}\in W^{1,2}(\Omega_h, \mathds{S}^2)$, we compute $\m_h^{n+1}$ by
  \begin{equation}
    \label{equ:midpoint-implicit-continuous}
    d_t\m_h^{n+1} - \alpha\bar{\m}_h^n\times\left(d_t\m_h^{n+1} + \tau d_{tt}\m_h^{n} \right) = -\bar{\m}_h^n\times\Delta_h\bar{\m}_h^n,
  \end{equation}
  where $\bar{\m}_h^n = \frac 1{2}(\m_h^{n+1} + \m_h^{n-1})$, and $\Delta_h$ represents the standard seven-point stencil of the Laplacian operator.
\end{algorithm}
The corresponding fully discrete version of the above \eqref{equ:midpoint-implicit-continuous} reads as
\begin{align}
  \label{equ:midpoint-implicit-discrete}
  \frac{\m_h^{n+1} - \m_h^{n-1}}{2k} - \alpha\frac{\m_h^{n+1} + \m_h^{n-1}}{2}\times\left(\frac{\m_h^{n+1} - \m_h^{n-1}}{2k} + \tau\frac{\m_h^{n+1} - 2\m_h^{n} + \m_h^{n-1}}{k^2}\right) \nonumber\\
  = -\frac{\m_h^{n+1} + \m_h^{n-1}}{2}\times\Delta_h\left(\frac{\m_h^{n+1} + \m_h^{n-1}}{2}\right).
\end{align}
Within three time steps, there have not been many direct discretization methods to get the second-order temporal accuracy. 
Due to the mid-point approximation feature, this implicit scheme is excellent in maintaining certain properties of the original system.
\begin{lem}
  Given $\abs{\m_h^0(\x_l)} = 1$, then the sequence $\{\m_h^n(\x_l)\}_{n\geq0}$ produced by \eqref{equ:midpoint-implicit-continuous} satisfies
  \begin{enumerate}[{\rm (i)}]
    \item \label{sublem1-1}$\abs{\m_h^n(\x_l)} = 1, \forall l \in L$;
    \item \label{sublem1-2}$\frac 1{2}D_t\norm{\nabla_h\m_h^{n+1}}_{2}^2 + \alpha\norm{d_t\m_h^{n+1}}_{2}^2 + \frac 1{2}\alpha\tau D_t^-\norm{d_t^+\m_h^{n}}_{2}^2 = 0$.
  \end{enumerate}
  \label{lem:lem2.1}
\end{lem}
\begin{proof}
On account of the initial condition \eqref{subequ:initial-condition-II}, we see that $\m^0(\x_l) = \m^{1}(\x_l)$ holding for all $l \in L$. Taking the vector inner product with \eqref{equ:midpoint-implicit-continuous} by $(\m_h^{n+1}(\x_l) + \m_h^{n-1}(x_l))$, it obvious that we can get
  \begin{equation*}
    |\m_h^{n+1}| = |\m_h^{n}| = \cdots = |\m_h^{1}| = |\m_h^0| = 1 ,
  \end{equation*}
in the point-wise sense. This confirms \eqref{sublem1-1}. In order to verify \eqref{sublem1-2}, we take inner product with \eqref{equ:midpoint-implicit-continuous} by $-\Delta_h\bar{\m}_h^n$ and get
  \begin{equation*}
    \frac 1{2}D_t\norm{\nabla_h\m_h^{n+1}}_{2}^2 - \alpha \langle \bar{\m}_h^n\times d_{t}\m_h^{n+1},-\Delta_h\bar{\m}_h^n \rangle - \alpha\tau \langle \m_h^n\times d_{tt}\m_h^n,-\Delta_h\bar{\m}_h^n \rangle = 0.
  \end{equation*}
Subsequently, taking inner products with $d_t\m_h^{n+1}$ and $d_{tt}\m_h^{n+1}$ separately leads to the following equalities:
  \begin{equation*}
    \norm{d_t\m_h^{n+1}}_{2}^2 - \alpha\tau \langle \m_h^n\times d_{tt}\m_h^n, d_t\m_h^{n+1} \rangle = - \langle \bar{\m}_h^n\times d_t\m_h^{n+1}, -\Delta_h\bar{\m}_h^n \rangle ,
  \end{equation*}
 and
  \begin{equation*}
    \frac 1{2}D_t^-\norm{d_t^+\m_h^{n}}_{2}^2 + \alpha \langle \m_h^n\times d_{tt}\m_h^n, d_t\m_h^{n+1} \rangle = - \langle \m_h^n\times d_{tt}\m_h^n, -\Delta_h\bar{\m}_h^n \rangle .
  \end{equation*}
  A combination of the above three identities yields \eqref{sublem1-2}.
\end{proof}

In \cref{lem:lem2.1}, taking $k\rightarrow 0$ gives
\begin{equation}
  \frac {\mathrm{d}}{\mathrm{d}t}\left(\frac 1{2}\norm{\nabla_h\m_h^{n+1}}_{2}^2 + \frac {\alpha\tau}{2} \norm{\partial_t\m_h^{n}}_{2}^2\right) = -\alpha\norm{\partial_t\m_h^{n+1}}_{2}^2,
\end{equation}
which is consistent with the continuous energy law \eqref{equ:continuous-energy-law}. Accordingly, in the absence of the external magnetic field, the discretized version energy dissipation law would be maintained with a modification
\begin{equation}
  E ( \m_h^{n+1} ,  \m_h^n )
     = \frac{\alpha \tau}{2} \Big\| \frac{\m_h^{n+1} - \m_h^n}{k} \Big\|_2^2
      + \frac{1}{4} ( \| \nabla_h \m_h^{n+1} \|_2^2
      + \| \nabla_h \m_h^n \|_2^2 ).
  \label{equ:modification-energy}
\end{equation}

\begin{thm} \label{thm: energy stability}
  Given $\m_h^{n-1}, \m_h^{n}, \m_h^{n+1}\in W^{1,2}(\Omega_h, \mathds{S}^2)$, we have a discrete energy dissipation law, for the modified energy \eqref{equ:modification-energy}:
  \begin{equation}
    E ( \m_h^{n+1} ,  \m_h^n ) \leq E ( \m_h^{n} , \m_h^{n-1} ).
    \label{eng est-0}
  \end{equation}
\end{thm}
\begin{proof}

Denote a discrete function
$$
  \mu^n := \alpha \Big( \frac{\m_h^{n+1} - \m_h^{n-1}}{2k} + \tau\frac{\m_h^{n+1} - 2\m_h^{n} + \m_h^{n-1}}{k^2} \Big) - \frac12 \Delta_h ( \m_h^{n+1} + \m_h^{n-1} ) .
$$
Taking a discrete inner product with~\eqref{equ:midpoint-implicit-continuous} by $\mu^n$ gives
\begin{align}
   &
   \frac{\alpha}{4 k^2} \langle \m_h^{n+1} - \m_h^{n-1} ,  \m_h^{n+1} - \m_h^{n-1} \rangle
  + \frac{\alpha \tau}{2 k^3}  \langle \m_h^{n+1} - \m_h^{n-1} ,
   \m_h^{n+1} - 2 \m_h^n + \m_h^{n-1} \rangle    \label{eng est-1}
\\
  & \qquad
    + \frac{\alpha}{4 k}  \Big\langle \m_h^{n+1} - \m_h^{n-1} ,
   - \Delta_h ( \m_h^{n+1} + \m_h^{n-1} )  \Big\rangle   \nonumber
\\
  &  \qquad
  = \Big\langle -\frac{\m_h^{n+1} + \m_h^{n-1}}{2}\times \mu^n , \mu^n \Big\rangle = 0 .
  \nonumber
\end{align}
Meanwhile, the following estimates are available:
\begin{align}
  &
   \langle \m_h^{n+1} - \m_h^{n-1} ,  \m_h^{n+1} - \m_h^{n-1} \rangle
   = \| \m_h^{n+1} - \m_h^{n-1} \|_2^2 \ge 0 ,   \label{eng est-2-1}
\\
  &
  \langle \m_h^{n+1} - \m_h^{n-1} ,
   \m_h^{n+1} - 2 \m_h^n + \m_h^{n-1} \rangle   \nonumber
\\
  = &
   \Big\langle (\m_h^{n+1} - \m_h^n) + (\m_h^n - \m_h^{n-1} ) ,
   (\m_h^{n+1} - \m_h^n) - (\m_h^n - \m_h^{n-1} ) \Big\rangle  ,  \nonumber
\\
  = &
   \| \m_h^{n+1} - \m_h^n \|_2^2 - \| \m_h^n - \m_h^{n-1} \|_2^2 ,
   \label{eng est-2-2}
\\
  &
  \Big\langle \m_h^{n+1} - \m_h^{n-1} ,
   - \Delta_h ( \m_h^{n+1} + \m_h^{n-1} )  \Big\rangle   \nonumber
\\
  =&
   \Big\langle \nabla_h ( \m_h^{n+1} - \m_h^{n-1} ) ,
   \nabla_h ( \m_h^{n+1} + \m_h^{n-1} )  \Big\rangle
  = \| \nabla_h \m_h^{n+1} \|_2^2 - \| \nabla_h \m_h^{n-1} \|_2^2
  \nonumber
\\
  =&
  ( \| \nabla_h \m_h^{n+1} \|_2^2  + \| \nabla_h \m_h^n \|_2^2 )
  - (  \| \nabla_h \m_h^n \|_2^2 + \| \nabla_h \m_h^{n-1} \|_2^2  )  .
  \label{eng est-2-3}
\end{align}
Going back to~\eqref{eng est-1}, we arrive at
\begin{align}
   &
    \frac{\alpha \tau}{2 k} \Big(  \Big\| \frac{\m_h^{n+1} - \m_h^n}{k} \Big\|_2^2
    - \Big\| \frac{\m_h^n - \m_h^{n-1}}{k}  \Big\|_2^2 \Big)
    \label{eng est-3}
\\
  & \qquad
      + \frac{1}{4 k} \Big(
 ( \| \nabla_h \m_h^{n+1} \|_2^2  + \| \nabla_h \m_h^n \|_2^2 )
  - (  \| \nabla_h \m_h^n \|_2^2 + \| \nabla_h \m_h^{n-1} \|_2^2  )  \Big)
     \nonumber
\\
  & \qquad
    = - \alpha  \Big\| \frac{\m_h^{n+1} - \m_h^{n-1}}{2k}  \Big\|_2^2 \le 0 ,
  \nonumber
\end{align}
which is exactly the energy dissipation estimate~\eqref{eng est-0}. This finishes the proof of Theorem~\ref{thm: energy stability}.
\end{proof}

Meanwhile, it is noticed that, given the initial profile of $\m$ at $t=0$, namely $\m^0$, an accurate approximation to $\m^1$ and $\m^2$ has to be made. In more details, an $O (k^2 + h^2)$ accuracy is required for both $\m^1$, $\m^2$ and $\frac{\m^1 - \m^0}{k}$, $\frac{\m^2 - \m^1}{k}$, which is needed in the convergence analysis.

The initial profile $\m^0$ could be taken as $\m^0 =  \m ( \cdot, 0)$. 
This in turn gives a trivial zero initial error for $\m^0$. For $\m^1$ and $\m^2$, a careful Taylor expansion reveals that
\begin{align}
  \m^1  =& \m^0 + k \partial_t \m^0 + \frac{k^2}{2} \partial_{tt} \m^0 + O (k^3)  \nonumber
\\
  = &
   \m^0  + \frac{k^2}{2} \partial_{tt} \m^0 + O (k^3) ,
  \label{initial-Taylor-1-1}
\\
  \m^2  = &
  \m^0 + 2 k \partial_t \m^0 + 2 k^2 \partial_{tt} \m^0 + O (k^3)  \nonumber
\\
  = &
  \m^0  + 2 k^2 \partial_{tt} \m^0 + O (k^3) ,
  \label{initial-Taylor-1-2}
\end{align}
in which the initial data~\eqref{subequ:initial-condition-II},  $\partial_t \m ( \cdot, 0) \equiv 0$, has been applied in the derivation. Therefore, an accurate approximation to $\m^1$ and $\m^2$ relies on a precise value of $\partial_{tt} \m$ at $t=0$. An evaluation of the original PDE~\eqref{subequ:iLLG-equ} implies that
\begin{align}
  \m^0 \times ( \partial_{tt} \m^0 ) = \frac{1}{\alpha \tau} \m^0 \times ( \Delta \m^0 + \mathbf{H}_{\mathrm{e}}^0 ) ,
  \label{initial-2}
\end{align}
in which the trivial initial data~\eqref{subequ:initial-condition-II} has been applied again. Meanwhile, motivated by the point-wise temporal differentiation identity
\begin{equation}
  \m \cdot \partial_{tt} \m = - ( \partial_t \m )^2 + \frac12 \partial_{tt} ( | \m |^2 )
  =  - ( \partial_t \m )^2 ,  \label{initial-3-1}
\end{equation}
and the fact that $| \m | \equiv 1$, we see that its evaluation at $t=0$ yields
\begin{equation}
  \m^0 \cdot \partial_{tt} \m^0 =  - ( \partial_t \m^0 )^2 = 0 .  \label{initial-3-2}
\end{equation}
Subsequently, a combination of~\eqref{initial-3-1} and \eqref{initial-3-2} uniquely determines $\partial_{tt} \m^0$:
\begin{equation}
  \partial_{tt} \m^0 = -\frac 1{\alpha\tau}\m^0\times( \m^0\times( \Delta \m^0 + \mathbf{H}_{\mathrm{e}}^0 ) ),
\end{equation}
and a substitution of this value into~\eqref{initial-Taylor-1-1}, \eqref{initial-Taylor-1-2} leads to an $O (k^3)$ approximation to $\m^1$ and $\m^2$.

Moreover, with spatial approximation introduced, an $O (k^2 + h^2)$ accuracy is obtained for both $\m^1$, $\m^2$ and $\frac{\m^1 - \m^0}{k}$, $\frac{\m^2 - \m^1}{k}$. This finishes the initialization process.

\section{Convergence analysis}
\label{sec:analysis}
The theoretical result concerning the convergence analysis is stated below.

\begin{thm}\label{cccthm2} Assume that the exact solution of \eqref{equ:iLL-equation} has the regularity $\m_e \in C^3 ([0,T]; [C^0(\bar{\Omega})]^3) \cap C^2([0,T]; [C^2(\bar{\Omega})]^3) \cap L^{\infty}([0,T]; [C^4(\bar{\Omega})]^3)$. Denote a nodal interpolation operator $\mathcal{P}_h$ such that $\mathcal{P}_h\m_h\in C^1(\Omega)$, and the numerical solution ${\m}_h^n$ ($n\ge0$) obtained from~\eqref{equ:midpoint-implicit-continuous} with the initial error satisfying $\| \e_h^p \|_2 +\|\nabla_h \e_h^p \|_2 = \mathcal{O} (k^2 + h^2)$, where $\e_h^p = \mathcal{P}_h \m_e (\cdot,t_p) - \m_h^p$, $p=0,1, 2$, and $\| \frac{\e_h^{q+1} - \e_h^q}{k} \|_2 = {\mathcal O} (k^2 + h^2)$, $q=0, 1$. 
Then the following convergence result holds for $2\leq n\leq \left\lfloor\frac{T}{k}\right\rfloor$ as $h,k\to0^+$:	\begin{align} \label{convergence-0}
	\| \mathcal{P}_h \m_e (\cdot,t_n) - \m_h^n \|_{2}
	+ \|\nabla_h ( \mathcal{P}_{h} \m_e (\cdot,t_n)- \m_h^n ) \|_2
	&\leq \mathcal{C}( k^2+h^2) ,
	\end{align}	
	in which the constant $\mathcal{C}>0$ is independent of $k$ and $h$.
\end{thm}

Before the rigorous proof is given, the following estimates are declared, which will be utilized in the convergence analysis. In the sequel, for simplicity of notation, we will use a uniform constant $\mathcal{C}$ to denote all the controllable constants throughout this part.

\begin{lem} [Discrete gradient acting on cross product] \label{lem 1} \cite{CHEN202155}
	For grid functions $\f_h$ and $\g_h$ over the uniform numerical grid, we have
	\begin{align}
	\|\nabla_h({\f}_h \times{\g}_h ) \|_2 &\leq \mathcal{C}\Big(\|{\f}_h\|_2 \cdot \|\nabla_h {\g}_h\|_\infty +\|{\g}_h\|_{\infty} \cdot \|\nabla_h{\f}_h\|_2 \Big) . \label{lem 1-0}
	\end{align}
\end{lem}

\begin{lem} [Point-wise product involved with second order temporal stencil] \label{lem 2}
	For grid functions $\f_h$ and $\g_h$ over the time domain, we have
	\begin{align}
	\frac{\f_h^{n+1} - 2 \f_h^n + \f_h^{n-1} }{k^2}  \cdot \g_h^n
	= & - \frac{ \f_h^n - \f_h^{n-1} }{k}  \cdot \frac{ \g_h^n - \g_h^{n-1} }{k}  \nonumber
\\
  &
	+ \frac{1}{k} \Big( \frac{\f_h^{n+1} - \f_h^n}{k}  \cdot \g_h^n
	-  \frac{\f_h^n - \f_h^{n-1}}{k}  \cdot \g_h^{n-1} \Big) . \label{lem 2-0}
	\end{align}
\end{lem}

	Now we proceed into the convergence estimate. First, we construct an approximate solution $\underline{\m}$:
	\begin{equation}
	\underline{\m} = \m_e + h^2 \m^{(1)} ,  \label{exact-1}
	\end{equation}
	in which the auxiliary field $\m^{(1)}$ satisfies the following Poisson equation
	\begin{align}  \label{exact-2}
	& \Delta \m^{(1)} = \hat{C}  \quad \mbox{with} \, \, \,
	\hat{C} = \frac{1}{| \Omega|} \int_{\partial \Omega} \,
	\partial_{\boldmath \nu}^3 \m_e \, \textrm{d} s ,  \\
	& \partial_z \m^{(1)} \mid_{z=0} = - \frac{1}{24} \partial_z^3 \m_e \mid_{z=0} ,  \quad
	\partial_z \m^{(1)} \mid_{z=1} = \frac{1}{24} \partial_z^3 \m_e \mid_{z=1} ,  \nonumber
	\end{align}
	with boundary conditions along $x$ and $y$ directions defined in a similar way.
	
	The purpose of such a construction will be illustrated later. Then we extend the approximate profile $\underline{\m}$ to the numerical ``ghost" points, according to the extrapolation formula:
	\begin{equation}
	\underline{\m}_{i,j,0}= \underline{\m}_{i,j,1} , \quad
	\underline{\m}_{i,j,nz+1} = \underline{\m}_{i,j,nz} ,  \label{exact-3}
	\end{equation}
	and the extrapolation for other boundaries can be formulated in the same manner. Subsequently, we prove that such an extrapolation yields a higher order $\mathcal{O}(h^5)$ approximation, instead of the standard $\mathcal{O}(h^3)$ accuracy. Also see the related works~\cite{STWW2003, Wang2000, Wang2004} in the existing literature.
	
	Performing a careful Taylor expansion for the exact solution around the boundary section $z=0$, combined with the mesh point values: ${z}_0 = - \frac12 h$, ${z}_1 = \frac12 h$, we get
	\begin{align}
	\m_e  ({x}_i, {y}_j, {z}_0 )
	&= \m_e ({x}_i, {y}_j, {z}_1 )
	- h \partial_z \m_e ({x}_i, {y}_j, 0 )
	- \frac{h^3}{24} \partial_z^3 \m_e ({x}_i, {y}_j, 0 )
	+   \mathcal{O}(h^5) \nonumber
	\\
	&= \m_e ({x}_i, {y}_j, {z}_1 )
	- \frac{h^3}{24} \partial_z^3 \m_e ({x}_i, {y}_j, 0 )
	+   \mathcal{O}(h^5) ,  \label{exact-4}
	\end{align}
	in which the homogenous boundary condition has been applied in the second step. A similar Taylor expansion for the constructed profile $\m^{(1)}$ reveals that
	\begin{align}
	\m^{(1)}  ({x}_i, {y}_j, {z}_0 )
	&= \m^{(1)} ({x}_i, {y}_j, {z}_1 )
	- h \partial_z \m^{(1)} ({x}_i, {y}_j, 0 )
	+   \mathcal{O}(h^3) \nonumber
	\\
	&= \m^{(1)} ({x}_i, {y}_j, {z}_1 )
	+ \frac{h}{24} \partial_z^3 \m_e ({x}_i, {y}_j, 0 )
	+   \mathcal{O}(h^3)  , \label{exact-5}
	\end{align}
	with the boundary condition in~\eqref{exact-2} applied. In turn, a substitution of~\eqref{exact-4}-\eqref{exact-5} into~\eqref{exact-1} indicates that
	\begin{equation}
	\underline{\m}  ({x}_i, {y}_j, {z}_0 )
	= \underline{\m} ({x}_i, {y}_j, {z}_1 )
	+   \mathcal{O}(h^5) .   \label{exact-6}
	\end{equation}
	In other words, the extrapolation formula~\eqref{exact-3} is indeed $\mathcal{O}(h^5)$ accurate.
	
	As a result of the boundary extrapolation estimate~\eqref{exact-6}, we see that the discrete Laplacian of $\underline{\m}$ yields the second-order accuracy at all the mesh points (including boundary points):
	\begin{equation}
	\Delta_h \underline{\m}_{i,j,k} = \Delta \m_e ({x}_i, {y}_j, {z}_k )
	+   \mathcal{O}(h^2) , \quad \forall 0 \le i \le nx+1, 0 \le j \le ny+1, 0 \le k \le nz+1.  \label{consistency-1}
	\end{equation}
	Moreover, a detailed calculation of Taylor expansion, in both time and space, leads to the following truncation error estimate:
	\begin{align}
	\frac{ \underline{\m}_h^{n+1} - \underline{\m}_h^{n-1} }{2 k} = &
	 \frac{\underline{\m}_h^{n+1} + \underline{\m}_h^{n-1}}{2} \times \Big( \alpha \frac{\underline{\m}_h^{n+1} - \underline{\m}_h^{n-1}}{2k} + \alpha \tau \frac{\underline{\m}_h^{n+1} - 2 \underline{\m}_h^{n} + \underline{\m}_h^{n-1}}{k^2}  \nonumber
 \\
   &
    - \Delta_h \Big(\frac{\underline{\m}_h^{n+1} + \underline{\m}_h^{n-1}}{2} \Big)  \Big) + \tau^n  ,
	\label{consistency-2} 	
	\end{align}
	where $\| \tau^n \|_2 \le \mathcal{C} (k^2+h^2)$. In addition, a higher order Taylor expansion in space and time reveals the following estimate for the discrete gradient of the truncation error, in both time and space:
\begin{equation}
  \| \nabla _h \tau^n \|_2 , \, \,
  \| \frac{\tau^n - \tau^{n-1}}{k} \|_2
  \le {\mathcal C} (k^2 + h^2) .    \label{truncation error-1}
\end{equation}
In fact, such a discrete $\| \cdot \|_{H_h^1}$ bound for the truncation comes from the regularity assumption for the exact solution, $\m_e \in C^3 ([0,T]; [C^0(\bar{\Omega})]^3) \cap C^2([0,T]; [C^2(\bar{\Omega})]^3) \cap L^{\infty}([0,T]; [C^4(\bar{\Omega})]^3)$, as stated in Theorem~\ref{cccthm2}, as well as the fact that $\m^{(1)} \in C^1 ([0, T]; [C^1(\bar{\Omega})]^3) \cap L^{\infty}([0,T]; [C^2(\bar{\Omega})]^3)$, as indicated by the Poisson equation~\eqref{exact-2}.	

We introduce the numerical error function ${\e}_h^n=\underline{\m}_h^n-\m_h^n$, instead of a direct comparison between the numerical solution and the exact solution. The error function between the numerical solution and the constructed solution ${\um}_h$ will be analyzed, due to its higher order consistency estimate~\eqref{exact-6} around the boundary. Therefore, a subtraction of \eqref{equ:midpoint-implicit-discrete} from the consistency estimate~\eqref{consistency-2} leads to the error function evolution system:
	\begin{align}
	&
	\frac{ \e_h^{n+1} - \e_h^{n-1}}{2 k} =
	\frac{\m_h^{n+1} + \m_h^{n-1}}{2} \times \tilde{\mmu}_h^n
	+ \frac{\e_h^{n+1} + \e_h^{n-1}}{2} \times \underline{\mmu}_h^n + \tau^n ,
	\label{error equation-1} 	
\\
  &
    \underline{\mmu}_h^n := \alpha \Big( \frac{\underline{\m}_h^{n+1} - \underline{\m}_h^{n-1}}{2k} + \tau \frac{\underline{\m}_h^{n+1} - 2 \underline{\m}_h^{n} + \underline{\m}_h^{n-1}}{k^2} \Big)
    - \Delta_h \Big(\frac{\underline{\m}_h^{n+1} + \underline{\m}_h^{n-1}}{2} \Big)  ,
    \label{error equation-2}
\\
  &
    \tilde{\mmu}_h^n := \alpha \Big( \frac{\e_h^{n+1} - \e_h^{n-1}}{2k} + \tau \frac{\e_h^{n+1} - 2  \e_h^{n} + \e_h^{n-1}}{k^2} \Big)
    - \Delta_h \Big(\frac{\e_h^{n+1} + \e_h^{n-1}}{2} \Big)  .   \label{error equation-3}
	\end{align}
	
	Before proceeding into the formal estimate, we establish a $W_h^\infty$ bound for $\underline{\mmu}_h^n$, which is based on the constructed approximate solution ${\um}$ (by~\eqref{error equation-2}). Because of the regularity for $\m_e$, the following bound is available:
	\begin{align}
	 \|  \underline{\mmu}_h^\ell \|_\infty  , \, \, \| \nabla_h \underline{\mmu}_h^\ell \|_\infty , \, \,
  \| \frac{\underline{\mmu}_h^n - \underline{\mmu}_h^{n-1}}{k} \|_\infty
  \le {\mathcal C} ,  \quad \ell = n, n-1 .   \label{bound-1}
	\end{align}

In addition, the following preliminary estimate will be useful in the convergence analysis.

\begin{lem} [A preliminary error estimate] \label{lem 3}
	We have
	\begin{align}
	\| \e_h^\ell \|_2^2 \le 2 \| \e_h^0 \|_2^2 + 2 T k \sum_{j=0}^{\ell-1}
	\| \frac{\e_h^{j+1} - \e_h^j}{k} \|_2^2 , \quad  \forall \ell \cdot k \le T . \label{lem 3-0}
	\end{align}
\end{lem}

\begin{proof}
We begin with the expansion:
\begin{align}
   \e_h^\ell  =  \e_h^0 + k \sum_{j=0}^{\ell-1} \frac{\e_h^{j+1} - \e_h^j}{k} ,
   \quad  \forall \ell \cdot k \le T . \label{lem 3-1}
\end{align}
In turn, a careful application of the Cauchy inequality reveals that
\begin{align}
  &
   \| \e_h^\ell \|_2^2 \le 2 \Big( \| \e_h^0 \|_2^2
   + k^2 \| \sum_{j=0}^{\ell-1} \frac{\e_h^{j+1} - \e_h^j}{k} \|_2^2 \Big) ,  \label{lem 3-2-1}
\\
  &
   k^2 \| \sum_{j=0}^{\ell-1} \frac{\e_h^{j+1} - \e_h^j}{k} \|_2^2
   \le k^2 \cdot \ell  \cdot \sum_{j=0}^{\ell-1} \| \frac{\e_h^{j+1} - \e_h^j}{k} \|_2^2
   \le T k  \sum_{j=0}^{\ell-1} \| \frac{\e_h^{j+1} - \e_h^j}{k} \|_2^2    ,  \label{lem 3-2-2}
\end{align}
in which the fact that $\ell \cdot k \le T$ has been applied. Therefore, a combination of \eqref{lem 3-2-1} and \eqref{lem 3-2-2} yields the desired estimate~\eqref{lem 3-0}. This completes the proof of Lemma~\ref{lem 3}.
\end{proof}

  Taking a discrete inner product with the numerical error equation \eqref{error equation-1} by $\tilde{\mmu}_h^n$ gives
	\begin{align}
	\frac{1}{2k} \langle \e_h^{n+1} - \e_h^{n-1} , \tilde{\mmu}_h^n \rangle =  &
	\langle \frac{\m_h^{n+1} + \m_h^{n-1}}{2} \times \tilde{\mmu}_h^n ,
	 \tilde{\mmu}_h^n \rangle  \nonumber
\\
    &
        + \langle \frac{\e_h^{n+1} + \e_h^{n-1}}{2} \times \underline{\mmu}_h^n ,
        \tilde{\mmu}_h^n \rangle
        + \langle \tau^n ,  \tilde{\mmu}_h^n \rangle .
	\label{convergence-1} 	
\end{align}
The analysis on the left hand side of~\eqref{convergence-1} is similar to the ones in~\eqref{eng est-2-1}-\eqref{eng est-2-3}:
\begin{align}
  &
  \frac{1}{2k} \langle \e_h^{n+1} - \e_h^{n-1}  , \tilde{\mmu}_h^n \rangle
  =  \frac{\alpha \tau}{2 k^3} \langle \e_h^{n+1} - \e_h^{n-1} ,
     \e_h^{n+1} - 2 \e_h^n + \e_h^{n-1} \rangle    \nonumber
\\
  &  \qquad \qquad \qquad
  + \frac{\alpha}{4 k^2} \langle \e_h^{n+1} - \e_h^{n-1} ,  \e_h^{n+1} - \e_h^{n-1} \rangle
  \nonumber
\\
  & \qquad \qquad \qquad
  + \frac{1}{4k}   \Big\langle \nabla_h ( \e_h^{n+1} - \e_h^{n-1} ) ,
   \nabla_h ( \e_h^{n+1} + \e_h^{n-1} )  \Big\rangle  , \label{convergence-2-1}
  \\
  &
   \langle \e_h^{n+1} - \e_h^{n-1} ,  \e_h^{n+1} - \e_h^{n-1} \rangle
   = \| \e^{n+1} - \e_h^{n-1} \|_2^2 \ge 0 ,   \label{convergence-2-2}
\\
  &
  \langle \e_h^{n+1} - \e_h^{n-1} ,
   \e_h^{n+1} - 2 \e_h^n + \e_h^{n-1} \rangle   \nonumber
\\
  = &
   \| \e_h^{n+1} - \e_h^n \|_2^2 - \| \e_h^n - \e_h^{n-1} \|_2^2 ,
   \label{convergence-2-3}
\\
  &
  \Big\langle \e_h^{n+1} - \e_h^{n-1} ,
   - \Delta_h ( \e_h^{n+1} + \e_h^{n-1} )  \Big\rangle   \nonumber
\\
  =&
   \Big\langle \nabla_h ( \e_h^{n+1} - \e_h^{n-1} ) ,
   \nabla_h ( \e_h^{n+1} + \e_h^{n-1} )  \Big\rangle
  = \| \nabla_h \e_h^{n+1} \|_2^2 - \| \nabla_h \e_h^{n-1} \|_2^2
  \nonumber
\\
  =&
  ( \| \nabla_h \e_h^{n+1} \|_2^2  + \| \nabla_h \e_h^n \|_2^2 )
  - (  \| \nabla_h \e_h^n \|_2^2 + \| \nabla_h \e_h^{n-1} \|_2^2  )  .
  \label{convergence-2-4}
\end{align}
This in turn leads to the following identity:
\begin{align}
  &
  \frac{1}{2k} \langle \e_h^{n+1} - \e_h^{n-1}  , \tilde{\mmu}_h^n \rangle
  =  \frac{1}{k} ( E_{\e, h}^{n+1} - E_{\e, h}^n )
   + \frac{\alpha}{4 k^2} \| \e_h^{n+1} - \e_h^{n-1} \|_2^2 ,   \label{convergence-3-1}
\\
  &
  E_{\e, h}^{n+1} =
  \frac{\alpha \tau}{2}  \| \frac{\e_h^{n+1} - \e_h^n}{k} \|_2^2
  + \frac14 ( \| \nabla_h \e_h^{n+1} \|_2^2  + \| \nabla_h \e_h^n \|_2^2 )  .
   \label{convergence-3-2}
\end{align}
The first term on the right hand side of~\eqref{convergence-1} vanishes, due to the fact that $\frac{\m_h^{n+1} + \m_h^{n-1}}{2} \times \tilde{\mmu}_h^n$ is orthogonal to $\tilde{\mmu}_h^n$, at a point-wise level:
\begin{equation}
   \langle \frac{\m_h^{n+1} + \m_h^{n-1}}{2} \times \tilde{\mmu}_h^n ,
	 \tilde{\mmu}_h^n \rangle = 0 .   \label{convergence-4}
\end{equation}
The second term on the right hand side of~\eqref{convergence-1} contains three parts:
\begin{align}
  &
   \langle \frac{\e_h^{n+1} + \e_h^{n-1}}{2} \times \underline{\mmu}_h^n ,
        \tilde{\mmu}_h^n \rangle  = I_1 + I_2 + I_3 ,  \label{convergence-5-1}
\\
  &
  I_1 =  \alpha \langle \frac{\e_h^{n+1} + \e_h^{n-1}}{2} \times \underline{\mmu}_h^n ,
         \frac{\e_h^{n+1} - \e_h^{n-1}}{2k}  \rangle ,   \label{convergence-5-2}
 \\
   &
  I_2 =  \alpha \tau \langle \frac{\e_h^{n+1} + \e_h^{n-1}}{2} \times \underline{\mmu}_h^n ,
         \frac{\e_h^{n+1} - 2 \e_h^{n} + \e_h^{n-1}}{k^2}  \rangle  ,  \label{convergence-5-3}
\\
  &
   I_3 =  \langle \frac{\e_h^{n+1} + \e_h^{n-1}}{2} \times \underline{\mmu}_h^n ,
    - \Delta_h \Big(\frac{\e_h^{n+1} + \e_h^{n-1}}{2} \Big)   \rangle  .  \label{convergence-5-4}
\end{align}
The first inner product, $I_1$, could be bounded in a straightforward way, with the help of discrete H\"older inequality:
\begin{align}
   I_1 =  & \alpha \langle \frac{\e_h^{n+1} + \e_h^{n-1}}{2} \times \underline{\mmu}_h^n ,
         \frac{\e_h^{n+1} - \e_h^{n-1}}{2k}  \rangle   \nonumber
\\
    \le & \frac{\alpha}{4}  \| \e_h^{n+1} + \e_h^{n-1}  \|_2
    \cdot  \| \underline{\mmu}_h^n  \|_\infty
    \cdot  \| \frac{\e_h^{n+1} - \e_h^{n-1}}{k} \|_2    \nonumber
\\
    \le  & {\mathcal C}  \| \e_h^{n+1} + \e_h^{n-1}  \|_2
    \cdot  \| \frac{\e_h^{n+1} - \e_h^{n-1}}{k} \|_2   \nonumber
\\
  \le &
    {\mathcal C}  ( \| \e_h^{n+1} \|_2^2 + \| \e_h^{n-1}  \|_2^2
    + \| \frac{\e_h^{n+1} - \e_h^{n-1}}{k} \|_2^2 ) .   \label{convergence-6}
\end{align}
For the second inner product, $I_2$,  we denote $\g_h^n := \frac{\e_h^{n+1} + \e_h^{n-1}}{2} \times \underline{\mmu}_h^n$. An application of point-wise identity~\eqref{lem 2-0} (in~\cref{lem 2}) reveals that
\begin{align}
  I_2 =  & \alpha \tau \langle \g_h^n ,
         \frac{\e_h^{n+1} - 2 \e_h^{n} + \e_h^{n-1}}{k^2}  \rangle  \nonumber
\\
	= & - \alpha \tau \langle \frac{ \e_h^n - \e_h^{n-1} }{k}  , \frac{ \g_h^n - \g_h^{n-1} }{k}  \rangle \nonumber
\\
  &
	+ \frac{\alpha \tau}{k} \Big( \langle \frac{\e_h^{n+1} - \e_h^n}{k} , \g_h^n \rangle
	-  \langle \frac{\e_h^n - \e_h^{n-1}}{k} , \g_h^{n-1}  \rangle \Big)  .
	\label{convergence-7-1} 	
\end{align}
Meanwhile, the following expansion is observed:
\begin{align}
   \frac{ \g_h^n - \g_h^{n-1} }{k}
   =   &  \frac14 ( \frac{\e_h^{n+1} - \e_h^n}{k} + \frac{\e_h^{n-1} - \e_h^{n-2}}{k} )
   \times ( \underline{\mmu}_h^n +  \underline{\mmu}_h^{n-1} )  \nonumber
\\
  &
    +  \frac{\e_h^{n+1} + \e_h^n + \e_h^{n-1} + \e_h^{n-2}}{4}
   \times \frac{ \underline{\mmu}_h^n  - \underline{\mmu}_h^{n-1} }{k} .
   \label{convergence-7-2} 	
\end{align}
This in turn indicates the associated estimate:
\begin{align}
   \| \frac{ \g_h^n - \g_h^{n-1} }{k}  \|_2
   \le   &  \frac14 ( \| \frac{\e_h^{n+1} - \e_h^n}{k} \|_2 + \| \frac{\e_h^{n-1} - \e_h^{n-2}}{k} \|_2 )
   \cdot ( \| \underline{\mmu}_h^n \|_\infty +  \| \underline{\mmu}_h^{n-1} \|_\infty )  \nonumber
\\
  &
    +  \frac{ \| \e_h^{n+1} \|_2 + \| \e_h^n \|_2 + \| \e_h^{n-1} \|_2 + \| \e_h^{n-2} \|_2 }{4}
   \cdot \| \frac{ \underline{\mmu}_h^n  - \underline{\mmu}_h^{n-1} }{k} \|_\infty  \nonumber
 \\
   \le &
   {\mathcal C}  \Big( \| \frac{\e_h^{n+1} - \e_h^n}{k} \|_2
   + \| \frac{\e_h^{n-1} - \e_h^{n-2}}{k} \|_2  \nonumber
 \\
   &  \qquad
   + \| \e_h^{n+1} \|_2 + \| \e_h^n \|_2 + \| \e_h^{n-1} \|_2 + \| \e_h^{n-2} \|_2 \Big)  ,
   \label{convergence-7-3} 	
\end{align}
in which the bound~\eqref{bound-1} has been applied. Going back to~\eqref{convergence-7-1}, we see that
\begin{align}
   - \alpha \tau \langle \frac{ \e_h^n - \e_h^{n-1} }{k}  , \frac{ \g_h^n - \g_h^{n-1} }{k}  \rangle
   \le &
    \alpha \tau \| \frac{ \e_h^n - \e_h^{n-1} }{k}  \|_2 \cdot
   \| \frac{ \g_h^n - \g_h^{n-1} }{k}  \|_2   \nonumber
 \\
   \le &
   {\mathcal C}  \Big( \| \frac{\e_h^{n+1} - \e_h^n}{k} \|_2
   + \| \frac{\e_h^{n-1} - \e_h^{n-2}}{k} \|_2
   + \| \e_h^{n+1} \|_2   \nonumber
 \\
   &  \quad
   + \| \e_h^n \|_2 + \| \e_h^{n-1} \|_2 + \| \e_h^{n-2} \|_2 \Big)
   \| \frac{ \e_h^n - \e_h^{n-1} }{k}  \|_2  \nonumber
\\
  \le &
   {\mathcal C}  \Big( \| \frac{\e_h^{n+1} - \e_h^n}{k} \|_2^2
   + \| \frac{\e_h^{n-1} - \e_h^{n-2}}{k} \|_2^2
   + \| \e_h^{n+1} \|_2^2   \nonumber
 \\
   &  \quad
   + \| \e_h^n \|_2^2 + \| \e_h^{n-1} \|_2^2 + \| \e_h^{n-2} \|_2^2
   + \| \frac{ \e_h^n - \e_h^{n-1} }{k}  \|_2^2 \Big) ,   \label{convergence-7-4}
\\
  I_2  \le &
    {\mathcal C}  \Big( \| \frac{\e_h^{n+1} - \e_h^n}{k} \|_2^2
   + \| \frac{\e_h^{n-1} - \e_h^{n-2}}{k} \|_2^2
   + \| \e_h^{n+1} \|_2^2   \nonumber
 \\
   &  \quad
   + \| \e_h^n \|_2^2 + \| \e_h^{n-1} \|_2^2 + \| \e_h^{n-2} \|_2^2
   + \| \frac{ \e_h^n - \e_h^{n-1} }{k}  \|_2^2 \Big)  \nonumber
\\
  &
  + \frac{\alpha \tau}{k} \Big( \langle \frac{\e_h^{n+1} - \e_h^n}{k} , \g_h^n \rangle
	-  \langle \frac{\e_h^n - \e_h^{n-1}}{k} , \g_h^{n-1}  \rangle \Big) .
    \label{convergence-7-5}    	
\end{align}
For the third inner product part, $I_3$, an application of summation by parts formula gives
\begin{align}
   I_3 =  & \langle \frac{\e_h^{n+1} + \e_h^{n-1}}{2} \times \underline{\mmu}_h^n ,
    - \Delta_h \Big(\frac{\e_h^{n+1} + \e_h^{n-1}}{2} \Big)   \rangle   \nonumber
\\
  = &
     \langle \nabla_h \Big( \frac{\e_h^{n+1} + \e_h^{n-1}}{2} \times \underline{\mmu}_h^n \Big) ,
    \nabla_h \Big(\frac{\e_h^{n+1} + \e_h^{n-1}}{2} \Big)   \rangle  .
     \label{convergence-8-1}
\end{align}
Meanwhile, we make use of the preliminary inequality~\eqref{lem 1-0} (in~\cref{lem 1}) and get
\begin{align}
         &
	\| \nabla_h \Big( \frac{\e_h^{n+1} + \e_h^{n-1}}{2} \times \underline{\mmu}_h^n \Big)  \|_2
	\nonumber
\\
	\le & \mathcal{C} \Big( \| \frac{\e_h^{n+1} + \e_h^{n-1}}{2} \|_2
	\cdot \|\nabla_h \underline{\mmu}_h^n \|_\infty +
	\| \underline{\mmu}_h^n \|_{\infty}
	\cdot \| \nabla_h ( \frac{\e_h^{n+1} + \e_h^{n-1}}{2} ) \|_2 \Big)  \nonumber
\\
     \le &
     \mathcal{C} \Big( \| \frac{\e_h^{n+1} + \e_h^{n-1}}{2} \|_2
	+  \| \nabla_h ( \frac{\e_h^{n+1} + \e_h^{n-1}}{2} ) \|_2 \Big)   \nonumber
\\
   \le  &
       \mathcal{C} \Big( \| \e_h^{n+1} \|_2 + \| \e_h^{n-1} \|_2
	+  \| \nabla_h \e_h^{n+1} \|_2 + \| \nabla_h \e_h^{n-1} \|_2 \Big)	. 	
	 \label{convergence-8-2}
\end{align}
Again, the bound~\eqref{bound-1} has been applied in the derivation. Therefore, the following estimate is available for $I_3$:
\begin{align}
   I_3 \le &
   \| \nabla_h \Big( \frac{\e_h^{n+1} + \e_h^{n-1}}{2} \times \underline{\mmu}_h^n \Big) \|_2
    \cdot \| \nabla_h \Big(\frac{\e_h^{n+1} + \e_h^{n-1}}{2} \Big) \|_2   \nonumber
\\
  \le &
   \mathcal{C} \Big( \| \e_h^{n+1} \|_2 + \| \e_h^{n-1} \|_2
	+  \| \nabla_h \e_h^{n+1} \|_2 + \| \nabla_h \e_h^{n-1} \|_2 \Big)	 \nonumber
\\
  &
	\cdot \Big(  \| \nabla_h \e_h^{n+1}\|_2 + \| \nabla_h \e_h^{n-1} \|_2 \Big)  \nonumber
\\
  \le &
   \mathcal{C} \Big( \| \e_h^{n+1} \|_2^2 + \| \e_h^{n-1} \|_2^2
	+  \| \nabla_h \e_h^{n+1} \|_2^2 + \| \nabla_h \e_h^{n-1} \|_2^2 \Big)  .
     \label{convergence-8-3}
\end{align}
The estimate of $I_3$ can also be obtained by a direct application of discrete H\"older inequality:
\begin{align}
  I_3 = & \langle \Big( \frac{\e_h^{n+1} + \e_h^{n-1}}{2} \times \nabla_h \underline{\mmu}_h^n \Big) ,
    \nabla_h \Big(\frac{\e_h^{n+1} + \e_h^{n-1}}{2} \Big)   \rangle \nonumber\\
  \leq & \frac 1{4}
   \| \e_h^{n+1} + \e_h^{n-1} \|_2 \cdot \|\nabla_h \underline{\mmu}_h^n \|_{\infty}
    \cdot \| \nabla_h (\e_h^{n+1} + \e_h^{n-1}) \|_2   \nonumber\\
    \leq &
   \mathcal{C} \Big( \| \e_h^{n+1} \|_2^2 + \| \e_h^{n-1} \|_2^2
	+  \| \nabla_h \e_h^{n+1} \|_2^2 + \| \nabla_h \e_h^{n-1} \|_2^2 \Big)  .
\label{convergence-8-4}
\end{align}

A substitution of~\eqref{convergence-6}, \eqref{convergence-7-5} and \eqref{convergence-8-4} into \eqref{convergence-5-1} yields the following bound:
\begin{align}
  &
   \langle \frac{\e_h^{n+1} + \e_h^{n-1}}{2} \times \underline{\mmu}_h^n ,
        \tilde{\mmu}_h^n \rangle  = I_1 + I_2 + I_3  \nonumber
\\
   \le  &
       \mathcal{C} \Big( \| \frac{\e_h^{n+1} - \e_h^n}{k} \|_2^2
   + \| \frac{\e_h^n - \e_h^{n-1}}{k} \|_2^2
   +  \| \frac{\e_h^{n-1} - \e_h^{n-2}}{k} \|_2^2   \nonumber
\\
  & \quad
     +  \| \e_h^{n+1} \|_2^2 +  \| \e_h^n \|_2^2 + \| \e_h^{n-1} \|_2^2
     +  \| \e_h^{n-2} \|_2 ^2
	+  \| \nabla_h \e_h^{n+1} \|_2^2 + \| \nabla_h \e_h^{n-1} \|_2^2 \Big)   \nonumber
\\
  &
  + \frac{\alpha \tau}{k} \Big( \langle \frac{\e_h^{n+1} - \e_h^n}{k} , \g_h^n \rangle
	-  \langle \frac{\e_h^n - \e_h^{n-1}}{k} , \g_h^{n-1}  \rangle \Big) .  	
       \label{convergence-9}
\end{align}

The third term on the right hand side of~\eqref{convergence-1} could be analyzed in a similar fashion:
\begin{align}
  &
   \langle \tau^n , \tilde{\mmu}_h^n \rangle  = I_4 + I_5 + I_6 ,  \label{convergence-10-1}
\\
  &
  I_4 =  \alpha  \langle \tau^n ,
         \frac{\e_h^{n+1} - \e_h^{n-1}}{2k}  \rangle ,   \quad
  I_5 =  \alpha \tau \langle \tau^n ,
         \frac{\e_h^{n+1} - 2 \e_h^{n} + \e_h^{n-1}}{k^2}  \rangle  ,  \label{convergence-10-2}
\\
  &
   I_6 =  \langle \tau^n ,
    - \Delta_h \Big(\frac{\e_h^{n+1} + \e_h^{n-1}}{2} \Big)   \rangle  ,   \label{convergence-10-3}
\end{align}
\begin{align}
   I_4 =  & \alpha \langle \tau^n ,
         \frac{\e_h^{n+1} - \e_h^{n-1}}{2k}  \rangle
          \le  \frac{\alpha}{2}  \| \tau^n  \|_2
    \cdot  \| \frac{\e_h^{n+1} - \e_h^{n-1}}{k} \|_2    \nonumber
\\
  \le &
    \frac{\alpha}{4}  ( \| \tau^n \|_2^2
    + \| \frac{\e_h^{n+1} - \e_h^{n-1}}{k} \|_2^2 ) ,  \label{convergence-11}
\end{align}
\begin{align}
  I_5 =  & \alpha \tau \langle \tau^n ,
         \frac{\e_h^{n+1} - 2 \e_h^{n} + \e_h^{n-1}}{k^2}  \rangle  \nonumber
\\
  = &
      - \alpha \tau \langle \frac{ \e_h^n - \e_h^{n-1} }{k}  ,
      \frac{ \tau^n - \tau^{n-1} }{k}  \rangle \nonumber
\\
  &
	+ \frac{\alpha \tau}{k} \Big( \langle \frac{\e_h^{n+1} - \e_h^n}{k} , \tau^n \rangle
	-  \langle \frac{\e_h^n - \e_h^{n-1}}{k} , \tau^{n-1}  \rangle \Big) ,
	\label{convergence-12-1} 	
\end{align}
\begin{align}
  &
      -  \langle \frac{ \e_h^n - \e_h^{n-1} }{k}  ,
      \frac{ \tau^n - \tau^{n-1} }{k}  \rangle
      \le   \| \frac{ \e_h^n - \e_h^{n-1} }{k} \|_2 \cdot
      \| \frac{ \tau^n - \tau^{n-1} }{k}  \|_2  \nonumber
\\
   \le &
      {\mathcal C} ( k^2 + h^2 ) \| \frac{ \e_h^n - \e_h^{n-1} }{k} \|_2
	\le {\mathcal C} ( k^4 + h^4 ) + \frac12 \| \frac{ \e_h^n - \e_h^{n-1} }{k} \|_2^2     ,
	\label{convergence-12-2} 	
\end{align}
\begin{align}
  I_5 \le  &
    {\mathcal C} ( k^4 + h^4 )
    + \frac{\alpha \tau}{2} \| \frac{ \e_h^n - \e_h^{n-1} }{k} \|_2^2  \nonumber
\\
  &
	+ \frac{\alpha \tau}{k} \Big( \langle \frac{\e_h^{n+1} - \e_h^n}{k} , \tau^n \rangle
	-  \langle \frac{\e_h^n - \e_h^{n-1}}{k} , \tau^{n-1}  \rangle \Big) ,
	\label{convergence-12-3} 	
\end{align}
\begin{align}
   I_6 =  & \langle \tau^n ,
    - \Delta_h \Big( \frac{\e_h^{n+1} + \e_h^{n-1}}{2} \Big)   \rangle
  =  \langle \nabla_h \tau^n,
     \nabla_h \Big( \frac{\e_h^{n+1} + \e_h^{n-1}}{2} \Big)   \rangle   \nonumber
\\
   \le &
     \|  \nabla_h \tau^n \|_2 \cdot
     \| \nabla_h \Big(\frac{\e_h^{n+1} + \e_h^{n-1}}{2} \Big) \|_2
  \le
    {\mathcal C} ( k^2 + h^2 )
     \| \nabla_h \Big(\frac{\e_h^{n+1} + \e_h^{n-1}}{2} \Big) \|_2   \nonumber
\\
     \le  &
     {\mathcal C} ( k^4 + h^4 )
     + \frac12 \Big( \| \nabla_h \e_h^{n+1} \|_2^2
    + \| \nabla_h \e_h^{n-1} \|_2^2  \Big) .
     \label{convergence-13}
\end{align}
Notice that the truncation error estimate~\eqref{truncation error-1} has been repeatedly applied in the derivation. Going back to~\eqref{convergence-10-1}, we obtain
\begin{align}
  &
   \langle \tau^n , \tilde{\mmu}_h^n \rangle  = I_4 + I_5 + I_6  \nonumber
\\
  \le &
   {\mathcal C} ( k^4 + h^4 )
   + \frac{\alpha}{2} \| \frac{ \e_h^{n+1} - \e_h^n }{k} \|_2^2
   + \frac{\alpha (\tau +1)}{2} \| \frac{ \e_h^n - \e_h^{n-1} }{k} \|_2^2    \nonumber
\\
  &
     + \frac12 \Big( \| \nabla_h \e_h^{n+1} \|_2^2
    + \| \nabla_h \e_h^{n-1} \|_2^2  \Big)    \nonumber
\\
  &
  + \frac{\alpha \tau}{k} \Big( \langle \frac{\e_h^{n+1} - \e_h^n}{k} , \tau^n \rangle
	-  \langle \frac{\e_h^n - \e_h^{n-1}}{k} , \tau^{n-1}  \rangle \Big) . \label{convergence-14}
\end{align}

Finally, a substitution of~\eqref{convergence-3-1}-\eqref{convergence-3-2}, \eqref{convergence-4}, \eqref{convergence-9} and \eqref{convergence-14} into \eqref{convergence-1} leads to the following inequality:
\begin{align}
  &
  \frac{1}{k} ( E_{\e, h}^{n+1} - E_{\e, h}^n )
   + \frac{\alpha}{4 k^2} \| \e_h^{n+1} - \e_h^{n-1} \|_2^2  \nonumber
\\
  \le &
   {\mathcal C} ( k^4 + h^4 )
   + \mathcal{C} \Big( \| \frac{\e_h^{n+1} - \e_h^n}{k} \|_2^2
   + \| \frac{\e_h^n - \e_h^{n-1}}{k} \|_2^2
   +  \| \frac{\e_h^{n-1} - \e_h^{n-2}}{k} \|_2^2   \nonumber
\\
  & \quad
     +  \| \e_h^{n+1} \|_2^2 +  \| \e_h^n \|_2^2 + \| \e_h^{n-1} \|_2^2
     +  \| \e_h^{n-2} \|_2 ^2
	+  \| \nabla_h \e_h^{n+1} \|_2^2 + \| \nabla_h \e_h^{n-1} \|_2^2 \Big)   \nonumber
\\
  &
  + \frac{\alpha \tau}{k} \Big( \langle \frac{\e_h^{n+1} - \e_h^n}{k} , \g_h^n + \tau^n \rangle
	-  \langle \frac{\e_h^n - \e_h^{n-1}}{k} , \g_h^{n-1}  + \tau^{n-1} \rangle \Big) .  	
       \label{convergence-15-1}
\end{align}
Subsequently, a summation in time yields
\begin{align}
   E_{\e, h}^{n+1}  \le &
   E_{\e, h}^2 + {\mathcal C} T ( k^4 + h^4 )
   + \mathcal{C} k \Big( \sum_{j=0}^n \| \frac{\e_h^{j+1} - \e_h^j}{k} \|_2^2
   +  \sum_{j=0}^{n+1} ( \| \e_h^j \|_2^2 + \| \nabla_h \e_h^j \|_2^2 )  \Big)   \nonumber
\\
  &
  +  \alpha \tau \Big( \langle \frac{\e_h^{n+1} - \e_h^n}{k} , \g_h^n + \tau^n \rangle
	-  \langle \frac{\e_h^2 - \e_h^1}{k} , \g_h^1  + \tau^1 \rangle \Big) .  	
       \label{convergence-15-2}
\end{align}
For the term $\alpha \tau \langle \frac{\e_h^{n+1} - \e_h^n}{k} , \g_h^n + \tau^n \rangle$, the following estimate could be derived
\begin{align}
  &
   \alpha \tau \langle \frac{\e_h^{n+1} - \e_h^n}{k} , \g_h^n + \tau^n \rangle
   \le  \frac{\alpha \tau}{4}  \| \frac{\e_h^{n+1} - \e_h^n}{k} \|_2^2
   + 2 \alpha \tau  ( \| \g_h^n \|_2^2 + \| \tau^n  \|_2^2  ) ,   \label{convergence-15-3}
\\
  &
  \| \g_h^n \|_2 = \| \frac{\e_h^{n+1} + \e_h^{n-1}}{2} \times \underline{\mmu}_h^n  \|_2
  \le \| \frac{\e_h^{n+1} + \e_h^{n-1}}{2} \|_2 \cdot \| \underline{\mmu}_h^n  \|_\infty  \nonumber
\\
  & \qquad \, \,
  \le {\mathcal C} \| \frac{\e_h^{n+1} + \e_h^{n-1}}{2} \|_2
  \le {\mathcal C} ( \| \e_h^{n+1} \|_2 + \| \e_h^{n-1} \|_2  ) ,
    \label{convergence-15-4}
\end{align}
in which the bound~\eqref{bound-1} has been used again. Then we get
\begin{align}
  \alpha \tau \langle \frac{\e_h^{n+1} - \e_h^n}{k} , \g_h^n + \tau^n \rangle
   \le  & \frac{\alpha \tau}{4}  \| \frac{\e_h^{n+1} - \e_h^n}{k} \|_2^2
   + 2 \alpha \tau   \| \tau^n  \|_2^2 \nonumber
 \\
   &
   + {\mathcal C} ( \| \e_h^{n+1} \|_2^2 + \| \e_h^{n-1} \|_2^2  )  \nonumber
\\
  \le &
  \frac12 E_{\e, h}^{n+1}
   + 2 \alpha \tau   \| \tau^n  \|_2^2
   + {\mathcal C} ( \| \e_h^{n+1} \|_2^2 + \| \e_h^{n-1} \|_2^2  )  ,
  \label{convergence-15-5}
\end{align}
in which the expansion identity, $E_{\e, h}^{n+1} =
\frac{\alpha \tau}{2}  \| \frac{\e_h^{n+1} - \e_h^n}{k} \|_2^2
  + \frac14 ( \| \nabla_h \e_h^{n+1} \|_2^2  + \| \nabla_h \e_h^n \|_2^2 )$ (given by~\eqref{convergence-3-2}), has been applied. Its substitution into~\eqref{convergence-15-2} gives
\begin{align}
   E_{\e, h}^{n+1}  \le &
   2 E_{\e, h}^2 + {\mathcal C} T ( k^4 + h^4 )
   + \mathcal{C} k \Big( \sum_{j=0}^n \| \frac{\e_h^{j+1} - \e_h^j}{k} \|_2^2
   +  \sum_{j=0}^{n+1} ( \| \e_h^j \|_2^2 + \| \nabla_h \e_h^j \|_2^2 ) \Big)    \nonumber
\\
  &
  + {\mathcal C} ( \| \e_h^{n+1} \|_2^2 + \| \e_h^{n-1} \|_2^2  )
  + 4 \alpha \tau   \| \tau^n  \|_2^2
  -  2 \alpha \tau \langle \frac{\e_h^2 - \e_h^1}{k} , \g_h^1  + \tau^1 \rangle .  	
       \label{convergence-15-6}
\end{align}
Moreover, an application of the preliminary error estimate~\eqref{lem 3-0} (in Lemma~\ref{lem 3}) leads to
\begin{align}
   E_{\e, h}^{n+1}  \le &
   2 E_{\e, h}^2 + {\mathcal C} T ( k^4 + h^4 )
   + \mathcal{C} (T^2 + 1) k \sum_{j=0}^n \| \frac{\e_h^{j+1} - \e_h^j}{k} \|_2^2
   + {\mathcal C}  T \| \e_h^0 \|_2^2  \nonumber
\\
  &
  + \mathcal{C} k  \sum_{j=0}^{n+1} \| \nabla_h \e_h^j \|_2^2
  + 4 \alpha \tau   \| \tau^n  \|_2^2
  -  2 \alpha \tau \langle \frac{\e_h^2 - \e_h^1}{k} , \g_h^1  + \tau^1 \rangle ,   	
       \label{convergence-15-7}
\end{align}
in which we have made use of the following fact:
	\begin{align}
	k \sum_{j=0}^{n+1} \| \e_h^j \|_2^2 	
	\le & k \cdot (n+1) \Big( 2 \| \e_h^0 \|_2^2 + 2 T k \sum_{j=0}^n
	\| \frac{\e_h^{j+1} - \e_h^j}{k} \|_2^2 \Big)  \nonumber
\\
  \le &
	 2 T \| \e_h^0 \|_2^2 + 2 T^2 k \sum_{j=0}^n
	\| \frac{\e_h^{j+1} - \e_h^j}{k} \|_2^2  .  \label{convergence-15-8}
	\end{align}
In addition, for the initial error quantities, the following estimates are available:
\begin{align}
  &
  E_{\e, h}^2 =
  \frac{\alpha \tau}{2}  \| \frac{\e_h^2 - \e_h^1}{k} \|_2^2
  + \frac14 ( \| \nabla_h \e_h^2 \|_2^2  + \| \nabla_h \e_h^1 \|_2^2 )
  \le {\mathcal C} (k^4 + h^4) ,
   \label{convergence-15-9-1}
\\
  &
   \| \e_h^0 \|_2^2 \le   {\mathcal C} (k^4 + h^4) ,  \label{convergence-15-9-2}
\\
  &
  4 \alpha \tau   \| \tau^n  \|_2^2   \le   {\mathcal C} (k^4 + h^4) ,
  \label{convergence-15-9-3}
\\
  &
   \| \g_h^1 \|_2 = \| \frac{\e_h^2 + \e_h^0}{2} \times \underline{\mmu}_h^1  \|_2
  \le \| \frac{\e_h^2 + \e_h^0}{2} \|_2 \cdot \| \underline{\mmu}_h^1  \|_\infty
   \le   {\mathcal C} (k^2 + h^2) ,
  \label{convergence-15-9-4}
\\
  &
  -  2 \alpha \tau \langle \frac{\e_h^2 - \e_h^1}{k} , \g_h^1  + \tau^1 \rangle
  \le  2 \alpha \tau \| \frac{\e_h^2 - \e_h^1}{k} \|_2 \cdot ( \| \g_h^1  \|_2 + \| \tau^1 \|_2 )
  \nonumber
\\
  & \qquad \qquad \qquad
   \le   {\mathcal C} (k^4 + h^4) ,
  \label{convergence-15-9-5}
\end{align}
which comes from the assumption in Theorem~\ref{cccthm2}. Then we arrive at
\begin{align}
   E_{\e, h}^{n+1}  \le &
     \mathcal{C} (T^2 + 1) k \sum_{j=0}^n \| \frac{\e_h^{j+1} - \e_h^j}{k} \|_2^2
   +  \mathcal{C} k \sum_{j=0}^{n+1} \| \nabla_h \e_h^j \|_2^2
   \nonumber
    + {\mathcal C} ( T + 1) ( k^4 + h^4 )     \nonumber
\\
   \le &
    {\mathcal C} ( T + 1) ( k^4 + h^4 )
   + \mathcal{C} (T^2 + 1) k \sum_{j=0}^n E_{\e, h}^{j+1} ,
       \label{convergence-15-10}
\end{align}
in which the fact that $E_{\e, h}^{j+1} =
  \frac{\alpha \tau}{2}  \| \frac{\e_h^{j+1} - \e_h^j}{k} \|_2^2
  + \frac14 ( \| \nabla_h \e_h^{j+1} \|_2^2 + \| \nabla_h \e_h^j \|_2^2 ) $, has been used. In turn, an application of discrete Gronwall inequality results in the desired convergence estimate:
	\begin{align}
	  &
	E_{\e, h}^{n+1} \le \mathcal{C}Te^{\mathcal{C}T} (k^4+h^4), \quad \text{for all } (n+1): n+1 \le \left\lfloor\frac{T}{k}\right\rfloor , \label{convergence-16-1}
\\
  &
	\| \frac{\e_h^{n+1} - \e_h^n}{k} \|_2 + \|\nabla_h \e_h^{n+1} \|_2
	\le \mathcal{C}(k^2+h^2) . \label{convergence-16-2}	
	\end{align}
Again, an application of the preliminary error estimate~\eqref{lem 3-0} (in Lemma~\ref{lem 3}) implies that
	\begin{align}
	\| \e_h^{n+1} \|_2^2 	
	\le & 2 \| \e_h^0 \|_2^2 + 2 T k \sum_{j=0}^n
	\| \frac{\e_h^{j+1} - \e_h^j}{k} \|_2^2 \le  \mathcal{C}(k^4+h^4) , \nonumber
\\
	\mbox{so that} \, \, \, \| \e_h^{n+1} \|_2 \le  & \mathcal{C}(k^2+h^2) . 	
	 \label{convergence-16-3}
	\end{align}
A combination of~\eqref{convergence-16-2} and \eqref{convergence-16-3} finishes the proof of Theorem~\ref{cccthm2}.

\section{A numerical solver for the nonlinear system}
\label{sec:solving}

It is clear that \Cref{alg:implicit-midpoint-algorithm} is a nonlinear scheme.
The following fixed-point iteration is employed to solve it.
\begin{algorithm}
  \label{alg:solving}
  Set $\m_h^{n+1,0} = 2\m_h^{n} - \m_h^{n-1}$ and $p = 0$.
  \begin{enumerate}[\rm(i)]
    \item Compute $\m_h^{n+1,p+1}$ such that
    \begin{multline}\label{equ:solving-in-iteration}
      \frac{\m_h^{n+1,p+1} - \m_h^{n-1}}{2k}= -\frac{\m_h^{n+1,p+1} + \m_h^{n-1}}{2}\times\Delta_h\left(\frac{\m_h^{n+1,p} + \m_h^{n-1}}{2}\right)\\+ \alpha\frac{\m_h^{n+1,p+1} + \m_h^{n-1}}{2}\times\left(\frac{\m_h^{n+1,p+1} - \m_h^{n-1}}{2k} + \tau\frac{\m_h^{n+1,p+1} - 2\m_h^{n} + \m_h^{n-1}}{k^2}\right). \\
    \end{multline}
    \item If $\| \m_h^{n+1, p+1} - \m_h^{n+1, p} \|_{2} \leq \epsilon$, then stop and set $\m_h^{n+1} = \m_h^{n+1, p+1}$.
    \item Set $p \leftarrow p+1$ and go to \rm{(i)}.
  \end{enumerate}
\end{algorithm}

Denote the operator
\begin{equation}
  \mathcal{L}^{p} = I - \alpha\m_h^{n-1}\times - \frac{2\alpha\tau}{k}\m_h^{n}\times - \frac{k}{2}\Delta_h(\m_h^{n+1,p}+\m_h^{n-1})\times,
\end{equation}
and make the fixed-point iteration solve the following equation
\begin{equation}
      \mathcal{L}^{p}\m_h^{n+1,p+1} = \m_h^{n-1} + \frac{2\alpha\tau}{k}\m_h^n\times\m_h^{n-1} - \frac{k}{2}\m_h^{n-1}\times\Delta_h(\m_h^{n+1,p}+\m_h^{n-1}) ,
\end{equation}
in its inner iteration. Under the condition $k \leq Ch^2$ with $C$ a constant, the following lemma confirms the convergence of \Cref{alg:solving}. 
For any $l\in L$ and owing to the property of $|\m_h(\x_l) | = 1$, it is clear that $0 < \| \m_h \|_{\infty} \leq 1$. For the discretized $\ell^2$ norm of $\hat{\m} \in C^3 ([0,T]; [C^0(\bar{\Omega})]^3) \cap C^2([0,T]; [C^2(\bar{\Omega})]^3) \cap L^{\infty}([0,T]; [C^4(\bar{\Omega})]^3)$, we have
\begin{equation}
  \|\nabla_h\hat{\m}\|_{2} \leq 2h^{-1}\|\hat{\m}\|_{2}.
\end{equation}
Then
\begin{multline}
  \|\Delta_h\hat{\m}\|^2_{2} = -\langle\nabla_h\hat{\m}, \nabla_h\Delta_h\hat{\m}\rangle \\ \leq \|\nabla_h\hat{\m}\|_{2}\|\nabla_h\Delta_h\hat{\m}\|_{2} \leq 2h^{-1}\|\nabla_h\hat{\m}\|_{2}\|\Delta_h\hat{\m}\|_{2},
\end{multline}
which in turn implies the following inverse inequality:
\begin{equation}
  \|\Delta_h\hat{\m}\|_{2} \leq 4h^{-2}\|\hat{\m}\|_{2}.
\end{equation}

\begin{lem}
  Let $|\m_h^{n-1}| = |\m_h^{n}| = 1$, there exists a constant $c_0$ such that $\|\m_h^{n-1}\|_{\infty}$, $\|\m_h^{n}\|_{\infty} \leq c_0$. The solution $\m_h^{n+1,p}$ calculated by \eqref{equ:solving-in-iteration} satisfies $|\m_h^{n+1,p}| = |\m_h^{n-1}|$ for $p = 1,2,\cdots$, which means that we can still find the constant $c_0 \leq 1$ satisfying $\|\m_h^{n+1,p}\|_{\infty} \leq c_0$. Then, for all $p \geq 1$, there exists a unique solution $\m_h^{n+1,p}$ in \eqref{equ:solving-in-iteration} and the following inequality is valid:
  \begin{equation}
    \|\m_h^{n+1,p+1} - \m_h^{n+1,p}\|_2 \leq 4 c_0 kh^{-2}\|\m^{n+1,p} - \m^{n+1,p-1}\|_2.
  \end{equation}
  \label{lem:iteration_convergence}
\end{lem}
\begin{proof}
  For any $\m_h\in\mathds{S}^2$, the following identity is clear:
  \begin{equation*}
    \langle \m_h, \mathcal{L}^{p}\m_h \rangle = 1,
  \end{equation*}
for all $p \geq 1$. Thus the operator $\mathcal{L}^{p}$ is positive definite for all $p \geq 1$, which provides the unique solvability of \eqref{equ:solving-in-iteration}.
   Taking the discrete inner product with \eqref{equ:solving-in-iteration} by $\m_h^{n+1,p+1} + \m_h^{n-1}$, we have $|\m_h^{n+1,p+1}| = 1$ in a point-wise sense, which means that the length of the magnetization is preserved at each step in the inner iteration. Thus, we can find a constant $c_0 \leq 1$ to control the $\ell^{\infty}$ norm of $\m_h^{n-1}$, $\m_h^{n}$ and $\m_h^{n+1,p}$ for $p=1,2,\cdots$ simultaneously.

  Subtraction of two subsequent equations in the fixed-point iteration yields
  \begin{multline*}
    \frac 1{2k}(\m_h^{n+1,p+1} - \m_h^{n+1,p}) = -\frac 1{4}(\m_h^{n+1,p+1} - \m_h^{n+1,p})\times\Delta_h\m_h^{n+1,p} \\- \frac 1{4}\m_h^{n+1,p}\times\Delta_h(\m_h^{n+1,p} - \m_h^{n+1,p-1}) \\- \frac 1{4}\m_h^n\times\Delta_h(\m^{n+1,p} - \m^{n+1,p-1}) \\- \frac 1{4}(\m_h^{n+1,p+1} - \m^{n+1,p})\times\Delta_h\m_h^{n-1} \\+ \frac{\alpha}{2k}\m_h^{n-1}\times(\m_h^{n+1,p+1} - \m_h^{n+1,p}) + \frac{\alpha\tau}{2k^2}\m_h^n\times(\m_h^{n+1,p+1} - \m_h^{n+1,p}).
  \end{multline*}
  Taking the inner product with $(\m_h^{n+1,p+1} - \m_h^{n+1,p})$ by the above equation produces
  \begin{align*}
    \|\m_h^{n+1,p+1} - \m_h^{n+1,p}\|_2 \leq &\frac{k}{2}\|\m_h^{n+1,p}\|_{\infty}\|\Delta_h(\m_h^{n+1,p}- \m_h^{n+1,p-1})\|_2 \\&+ \frac{k}{2}\|\m_h^n\|_{\infty}\|\Delta_h(\m_h^{n+1,p} - \m_h^{n+1,p-1})\|_2\\
    \leq &c_0 k\|\Delta_h(\m^{n+1,p} - \m^{n+1,p-1})\|_2.
  \end{align*}
In turn, the convergence result becomes
  \begin{equation}
    \|\m_h^{n+1,p+1} - \m_h^{n+1,p}\|_2 \leq 4 c_0 kh^{-2}\|\m_h^{n+1,p} - \m_h^{n+1,p-1}\|_2,
  \end{equation}
which completes the proof of Lemma~\ref{lem:iteration_convergence}.
\end{proof}

\section{Numerical experiments}
\label{sec:experiments}
\subsection{Accuracy tests}
Consider the 1-D iLLG equation
\begin{equation*}
  \partial_t\m = -\m\times\partial_{xx}\m + \alpha\m\times(\partial_t\m + \tau\partial_{tt}\m) + \mathbf{f} .
\end{equation*}
The exact solution is chosen to be $\m_{\mathrm{e}} = (\cos(\bar x)\sin(t^2), \sin(\bar x)\sin(t^2), \cos(t^2))^T$ with $\bar x = x^2(1-x)^2$, and the forcing term is given by $\mathbf{f} = \partial_t\m_{\mathrm{e}} + \m_{\mathrm{e}}\times\partial_{xx}\m_{\mathrm{e}}- \alpha\m_{\mathrm{e}}\times(\partial_t\m_{\mathrm{e}} + \tau\partial_{tt}\m_{\mathrm{e}})$. Fixing the tolerance $\epsilon = \textrm{1.0e-07}$ for the fixed-point iteration, we record the discrete $\ell^2$ and $\ell^{\infty}$ errors between the exact solution and numerical solution with a sequence of temporal step-size and spatial mesh-size. The parameters in the above 1-D equation are set as: $\alpha = 0.1$, $\tau = 10.0$, and the final time $T = 0.01$. The temporal step-sizes and spatial mesh-sizes are listed in the \Cref{tab:1dtime} and \Cref{tab:1dspace}.
\begin{table}[ht]
  \centering
  \caption{The discrete $\ell^2$ and $\ell^{\infty}$ errors in terms of the temporal step-size. The spatial mesh-size is fixed as $h = 0.001$ over $\Omega = (0,1)$ and the final time is $T = 0.01$.}
  \begin{tabular}{||c|c|c||}
    \hline
    $k$ & $\|\m_h - \m_e\|_2$ & $\|\m_h - \m_e\|_{\infty}$\\
    \hline
    T/40 & 1.2500e-11 & 1.2584e-11 \\
    T/60 & 5.6887e-12 & 5.6024e-12 \\
    T/80 & 3.2008e-12 & 3.1525e-12 \\
    T/100 & 2.0487e-12 & 2.0174e-12 \\
    \hline
    order & 1.97 & 2.00 \\
    \hline
  \end{tabular}
  \label{tab:1dtime}
\end{table}
\begin{table}[ht]
  \centering
  \caption{The discrete $\ell^2$ and $\ell^{\infty}$ error in terms of the spatial mesh-size. The parameters are set as: the temporal step-size $k = \textrm{2.0e-06}$, $\Omega = (0,1)$ and the final time $T = 0.5$.}
  \begin{tabular}{||c|c|c||}
    \hline
    $h$ & $\|\m_h - \m_e\|_2$ & $\|\m_h - \m_e\|_{\infty}$\\
    \hline
    1/20 & 1.9742e-05 & 2.5320e-05 \\
    1/40 & 4.9846e-06 & 6.3459e-06 \\
    1/60 & 2.2340e-06 & 2.8201e-06 \\
    1/80 & 1.2720e-06 & 1.5853e-06 \\
    \hline
    order & 1.98 & 2.00 \\
    \hline
  \end{tabular}
  \label{tab:1dspace}
\end{table}

In addition, the 3-D iLLG equation is also considered:
\begin{equation*}
  \partial_t\m = -\m\times\Delta\m + \alpha\m\times(\partial_t\m + \tau\partial_{tt}\m) + \mathbf{f} .
\end{equation*}
The exact solution is chosen to be $\m_{\mathrm{e}} = (\cos(\bar x\bar y\bar z)\sin(t^2), \sin(\bar x\bar y\bar z)\sin(t^2)), \cos(t^2))^T$ with $\bar y = y^2(1-y)^2$ and $\bar z = z^2(1-z)^2$, and the forcing term $\mathbf{f} = \partial_t\m_{\mathrm{e}} + \m_{\mathrm{e}}\times\Delta\m_{\mathrm{e}}- \alpha\m_{\mathrm{e}}\times(\partial_t\m_{\mathrm{e}} + \tau\partial_{tt}\m_{\mathrm{e}})$. Similarly, we record the discrete $\ell^2$ and $\ell^{\infty}$ errors between exact and numerical solutions with a sequence of temporal step-sizes and spatial mesh-sizes. The corresponding parameters are set as: $\alpha = 0.01$ and $\tau = 1000.0$. Besides, the final time of this simulation is given by $T = 0.01$, with the temporal step-size and spatial mesh-size listed in  \Cref{tab:3dtime} and \Cref{tab:3dspace}.

\begin{table}[ht]
  \centering
  \caption{The discrete $\ell^2$ and $\ell^{\infty}$ errors in terms of the temporal step-size. The spatial mesh-size is fixed as $h = 0.001$ and final time is $T = 0.01$.}
  \begin{tabular}{||c|c|c||}
    \hline
    $k$ & $\|\m_h - \m_e\|_2$ & $\|\m_h - \m_e\|_{\infty}$\\
    \hline
    T/100 & 1.2678e-05 & 1.2765e-05 \\
    T/120 & 8.8067e-06 & 8.8830e-06 \\
    T/140 & 6.4725e-06 & 6.5419e-06 \\
    T/160 & 4.9576e-06 & 5.0224e-06 \\
    \hline
    order & 2.00 & 1.98 \\
    \hline
  \end{tabular}
  \label{tab:3dtime}
\end{table}
\begin{table}[ht]
  \centering
  \caption{The discrete $\ell^2$ and $\ell^{\infty}$ errors in terms of spatial mesh-size. The temporal step-size is fixed as $k = \textrm{2.0e-06}$.}
  \begin{tabular}{||c|c|c||}
    \hline
    $h$ & $\|\m_h - \m_e\|_2$ & $\|\m_h - \m_e\|_{\infty}$\\
    \hline
    1/8 & 1.4392e-07 & 3.4940e-07 \\
    1/10 & 9.6832e-08 & 2.2864e-07 \\
    1/12 & 6.9825e-08 & 1.6079e-07 \\
    1/14 & 5.2828e-08 & 1.1895e-07 \\
    \hline
    order & 1.79 & 1.92 \\
    \hline
  \end{tabular}
  \label{tab:3dspace}
\end{table}

\subsection{Micromagnetics tests}
The inertial effect can be observed during the relaxation of a system with a non-equilibrium initialization. To visualize this, we conduct micromagnetics simulations for both the LLG equation and the iLLG equation.

In the following simulations, a 3-D domain $\Omega = [0,1]\times[0,1]\times[0,0.4]$ is uniformly discretized into $10\times10\times4$ cells, with uniform initialization $\m^{0} = (\sqrt{2}/2, \sqrt{2}/2, 0)^T$. For comparison, the LLG equation is discretized by the mid-point scheme proposed in \cite{bartels2006convergence} with the fixed-point iteration solver proposed in this work. The damping parameter is $\alpha = 0.5$ and the field is fixed as $\mathbf{H}_{\mathrm{e}} = (10, 0, 0)^T$, which indicates that the system shall converge to $\m = (1, 0, 0)^T$. Here the relaxation of the magnetization behavior controlled by the LLG equation is visualized in \Cref{fig:LLGSAM}.
\begin{figure}[ht]
  \centering
  \includegraphics[width=5in]{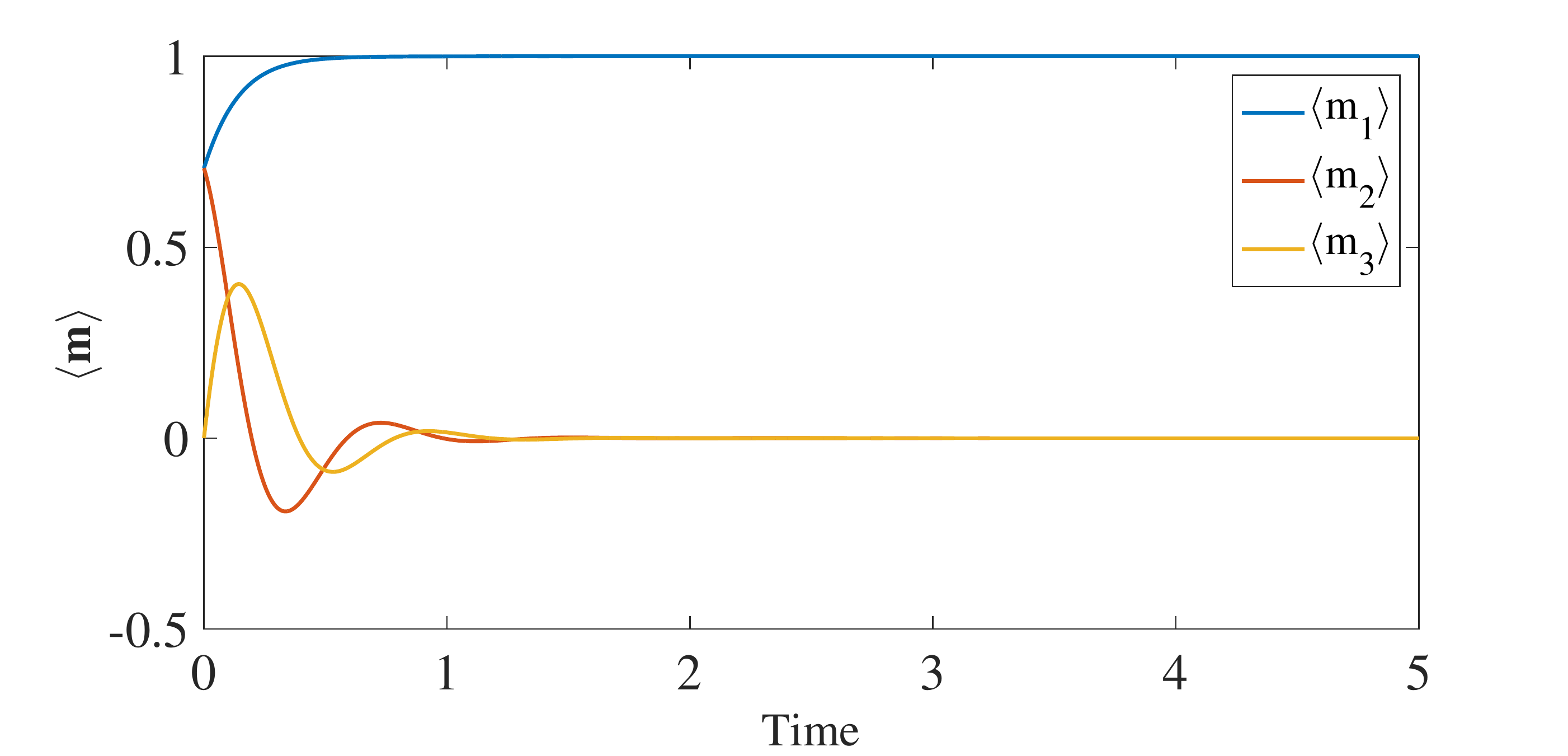}
  \caption{The relaxation of the spatially averaged magnetization controlled by the LLG equation. The final time is $T = 5.0$ with $k=0.001$, and the damping parameter is $\alpha = 0.5$.}
  \label{fig:LLGSAM}
\end{figure}

As for the counterpart of the LLG equation, wtih a given reference field $\mathbf{H}_{\mathrm{e}}$, the discrete energy of the iLLG equation becomes
\begin{multline}
  \mathbf{E}[\m^{n+1},\m^n] = \frac{1}{4} ( \| \nabla_h \m_h^{n+1} \|_2^2
      + \| \nabla_h \m_h^n \|_2^2 )
      + \\ \frac{\alpha \tau}{2} \Big\| \frac{\m_h^{n+1} - \m_h^n}{k} \Big\|_2^2 - \frac 1{2}\langle \m_h^{n+1}+\m_h^{n}, \mathbf{H}_{\mathrm{e}} \rangle.
\end{multline}
Setting the inertial parameter $\tau = 1.0$, the spatially averaged magnetization is recorded to depict the inertial effect in \Cref{fig:iLLGSAM}(a). Meanwhile, the energy decay is also numerically verified by \Cref{fig:iLLGSAM}(b). The inertial effect is observed at shorter timescales for magnetization dynamics during the relaxation of the system with a non-equilibrium initialization.

\begin{figure}[ht]
  \centering
  \subfloat[Spatially averaged magnetization]{\includegraphics[width=5in]{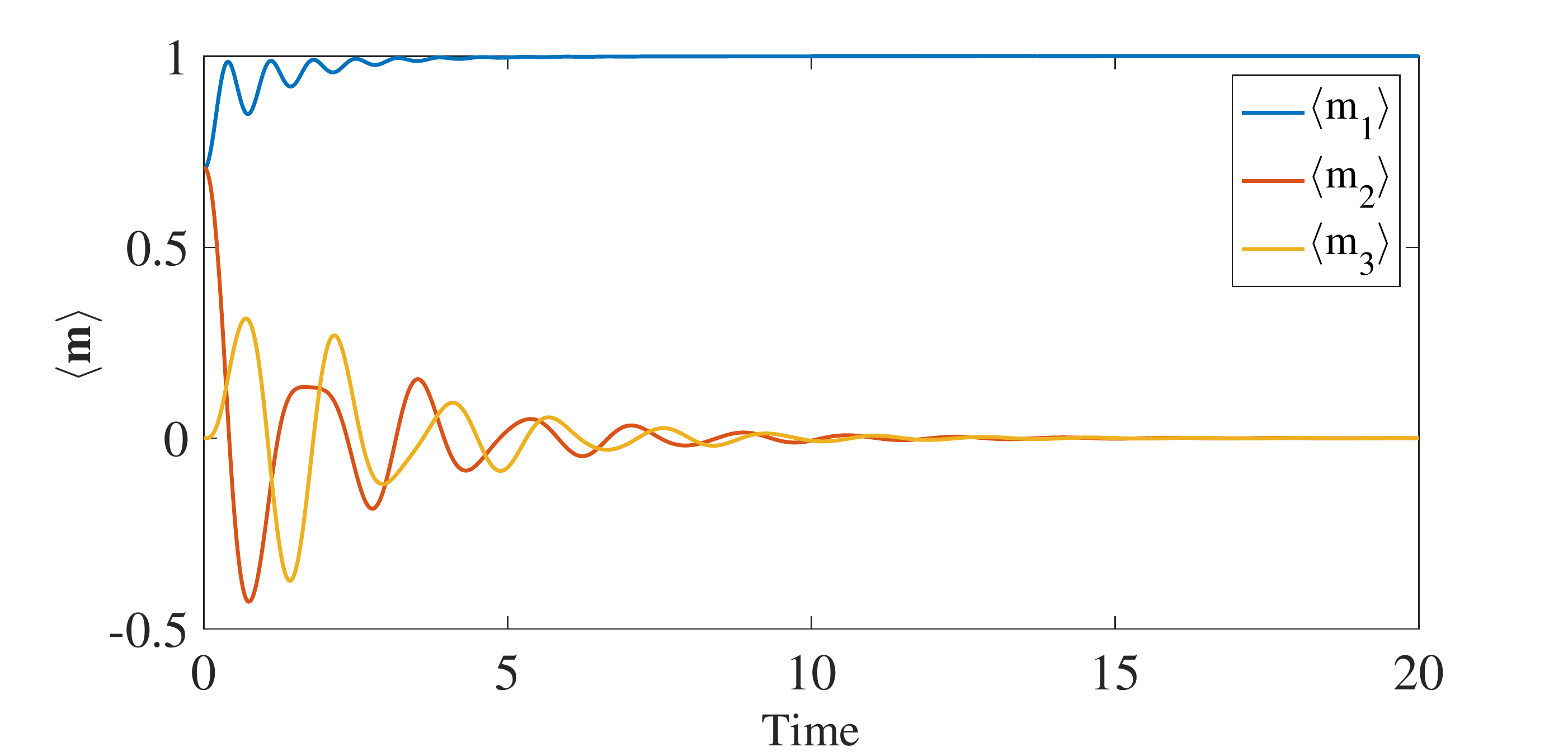}}
  \quad
  \subfloat[Energy decay]{\includegraphics[width=5in]{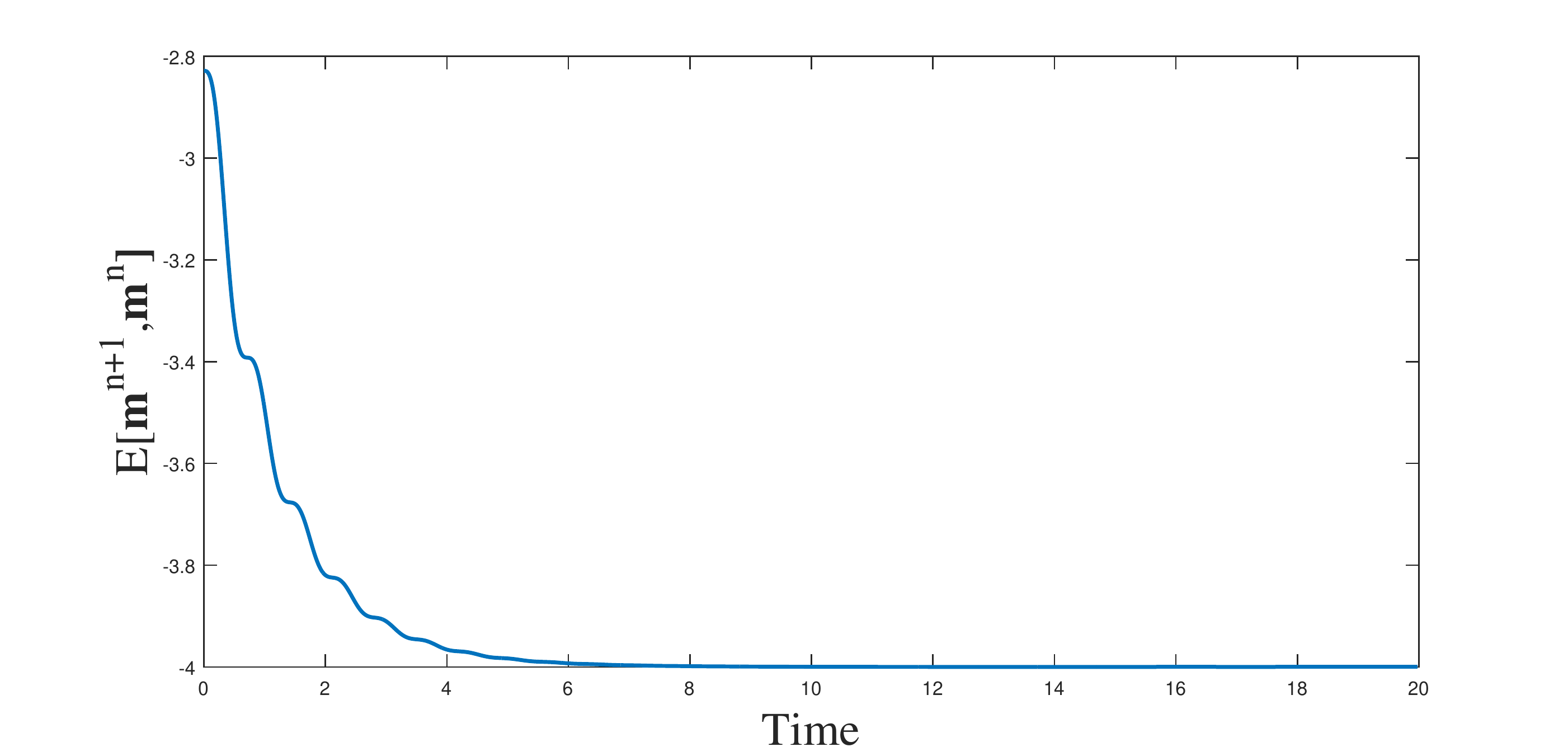}}
  \caption{The spatially averaged magnetization (A) and the energy evolution (B) in the iLLG equation. Parameters setting: $T = 20$, $k=0.02$, $\tau = 1.0$ and $\alpha = 0.5$.}
  \label{fig:iLLGSAM}
\end{figure}

Furthermore, the inertial effect can also be activated by an external perturbation applied to an equilibrium state. Here we set the damping parameter $\alpha = 0.02$ and $\tau = 0.5$, then the time step-size must be reduced to 0.001 with $T = 3.0$. For the equilibrium state $\m^0 = (1, 0, 0)^T$, the perturbation $4.0\times\sin(2\pi ft)$ is applied along $y$ direction over the time interval $[0, 0.05]$, with $f = 20$. The relaxation of the iLLG equation, revealed by the evolution of the spatially averaged magnetization, is visualized in \Cref{fig:iLLGSAMa}.
\begin{figure}[htbp]
  \centering
  \includegraphics[width=5in]{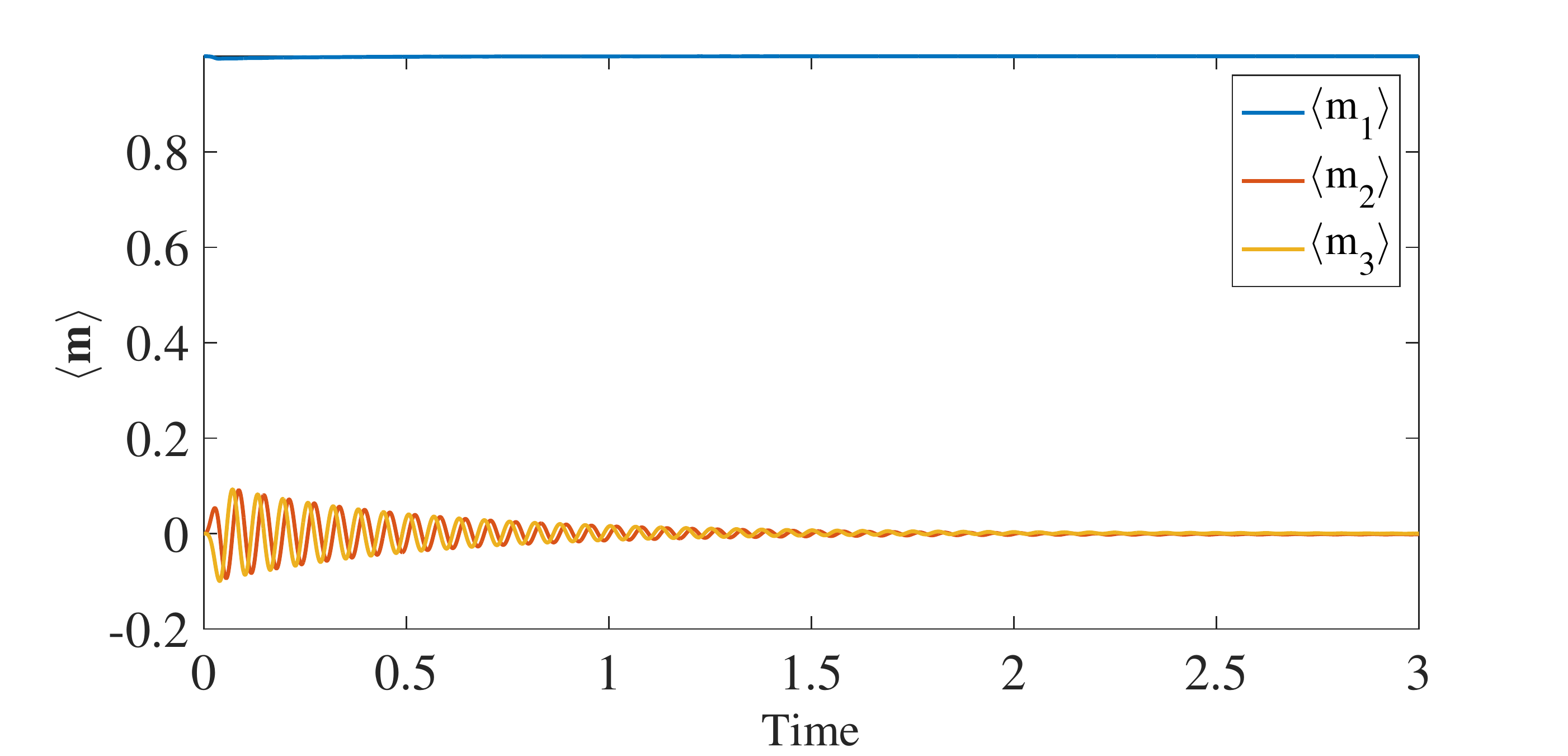}
  \caption{The response of the spatially averaged magnetization for the magnetic perturbation in the presence of the inertial effect. For the equilibrium initialization $\m^0 = (1, 0, 0)^T$, a perturbation $4.0\times\sin(2\pi ft)$ is applied along $y$ direction during time interval $0\sim0.05$, with $f = 20$. The basic simulation parameters are: $\alpha = 0.02$, $\tau = 0.5$, $T = 3.0$ and $k=0.001$.}
  \label{fig:iLLGSAMa}
\end{figure}

\section{Conclusion}
\label{sec:conclusion}

In this work, an implicit mid-point scheme, with three time steps, is proposed to solve the inertial Landau-Lifshitz-Gilbert equation. This algorithm preserves the properties of magnetization dynamics, such as the energy decay and the constant length of magnetization and is proved to be second-order accurate in both space and time. In the convergence analysis, we first construct a second-order approximation $\underline{\m}$ of the exact solution. It is found that $\underline{\m}$ produces $\mathcal{O}(h^5)$ accuracy at the mesh points around the boundary sections, which simplifies the estimation at boundary points. Then, by analyzing the error function between the numerical solution and the constructed solution $\underline{\m}_h$, we derive the convergence result in the $H^1(\Omega_T)$ norm. Furthermore, a fixed-point iteration method is proposed to solve this implicit nonlinear scheme under the time-step restriction $k \leq Ch^2$. Numerical results confirm the theoretical analysis and clearly show the unique inertial effect in micromagnetics simulations.

\section*{Acknowledgments}
This work is supported in part by the grant NSFC 11971021 (J. Chen), the program of China Scholarships Council No. 202106920036 (P. Li), the grant NSF DMS-2012269 (C. Wang).
\bibliographystyle{amsplain}
\bibliography{references}

\end{document}